\documentclass[MNA]{mna} %

\SHORTTITLE{Neural Coding as a Statistical Testing Problem}

\TITLE{Neural Coding as a Statistical Testing Problem
    } 

\AUTHORS{%
  Guilherme~Ost\footnote{Universidade Federal do Rio de Janeiro.
    \EMAIL{guilhermeost@im.ufrj.br}}
  \and 
 Patricia~Reynaud-Bouret\footnote{Université Côte d’Azur, CNRS, LJAD, France.
 	\EMAIL{Patricia.Reynaud-Bouret@univ-cotedazur.fr}}}

\SHORTAUTHOR{G. Ost and P. Reynaud-Bouret}


\AMSSUBJ{ 62G10; 62C20; 62M02; 62P10; 92C20} 

\SUBMITTED{September 5, 2022} 
\ACCEPTED{November 14, 2023} 

\VOLUME{3}
\YEAR{2023}
\PAPERNUM{4}
\DOI{10.46298/mna.10002}

\ABSTRACT{We take the testing perspective to understand what the minimal discrimination time between two stimuli is for different types of rate coding neurons. Our main goal is to describe the testing abilities of two different encoding systems: place cells and grid cells. In particular, we show, through the notion of adaptation, that a fixed place cell  system can have a  minimum discrimination time that decreases when the stimuli are further away. This could be a considerable advantage for the place cell  system that could complement the grid cell system, which is able to discriminate stimuli that are much closer than place cells.}

\usepackage{amssymb}
\usepackage{amsthm}
\usepackage{stmaryrd}
\usepackage{mathtools}

\def\bE{{\mathbb E}}
\def\bP{{\mathbb P}}

\def\bR{{\mathbb R}}

\def\bN{{\mathbb N}}

\def\bS{\mathbb{S}}

\def\cP{{\mathcal P}}
\def\cG{{\mathcal G}}
\def\cF {\mathcal{F}}
\def\cC {\mathcal{C}}
\def\cS {\mathcal{S}}

 \def\llb{\llbracket}
\def\rrb{\rrbracket}

\def\modxu{x \ \text{mod} \ u}
\def\modsu{s \ \text{mod} \ u}

\def\modsi{s \ \text{mod} \ \lambda_i}
\def\mod{\text{mod}}
\def\modS2{s_2 \ \text{mod} \ 1}

\begin{document}



\section{Introduction}

Grid cells are particular neurons in medial enthorinal cortex \cite{McNaughton_Moser_nature_review_06}  that exhibit a periodic spatial firing pattern. Whereas place cells in the hippocampus fire at a given location and appear to be associated with an allocentric representation, grid cells fire at each node of a hexagonal lattice and appear to be involved in a self-localization representation  \cite{Moser_Kropff_Moser_08}. The grid cell system not only encodes spatial position but also direction and velocity \cite{Sargolini_2006} or sounds \cite{Aronov_2017} and could even be used by mammals to encode episodic memories \cite{Buzaki_Moser_2013}. Organized in modules (one module being dedicated to one lattice  scale), the grid cells in each module have a uniform distribution, while the progression of scales between modules appears to be geometric \cite{stensola_et_al_12}.

Since their discovery in 2005 \cite{Moser_2005}, which earned O'Keefe and the Mosers the Nobel Prize, grid cells have been intensively studied from a theoretical point of view in relation to place cells \cite{Okeefe_1971}.
Some authors (see for example \cite{Schwarz_koy_2019}) are interested in how neural networks can generate such patterns, while others try to explain the hexagonal lattice \cite{chen} or the exact geometric progression of scales \cite{Wei_2015,Stemmler_2015}. 

In particular, one of the main focuses has been the encoding capacity of the system. More specifically, the authors focused on a statistical capacity measure, namely the Fisher information, because of its link with estimation. Indeed in Statistics,  Cramér-Rao bound states that the $L_2$-error
(i.e., $\sqrt{\mathbb{E}(|\hat{s}_n-s|^2)})$ of an unbiased estimator $\hat{s}_n$ (i.e. , $\mathbb{E}(\hat{s}_n)=s$)) of a given quantity $s$ is lower bounded by $I_n(s)^{-1/2}$, where $I_n(s)$ is  the Fisher information. Moreover, this lower bound is generally achieved by Maximum Likelihood Estimators (MLE), at least asymptotically in a context of $n$ i.i.d. observations \cite{CoxHink74}.
Informally, the Cramér-Rao bound is interpreted as follows: there is a "best" estimator (which would be the MLE) that would achieve the smallest error $I_n(s)^{-1/2}$. So if  a code refers to a model that describes the  precise influence of $s$ on the spike trains emitted by the neural cells, Cram\`er-Rao bounds paves the path of finding the best code as the code that maximizes the Fisher information.

Initially, \cite{Brunel_1998} worked on the relationship between mutual information and Fisher information for place cells and other types of neurons with receptive fields. They showed in particular that if $s$ represents a  position, then $I_n(s)$ grows linearly with $n$, for a particular code where the $n$ place cells responses to position $s$ are i.i.d. This leads to an error in $n^{-1/2}$.
Then Fiete and her co-authors \cite{Fiete_2008,Fiete_2011} showed that for a given number $n$ of neurons, grid cells can encode many more positions and that $I_n(s)$ grows exponentially with $n$ the number of place cells. 
In particular, this means that the accuracy of the position estimator that can be done with $n$ grid cells is much more precise than the accuracy that can be done with $n$ place cells (see also these related works \cite{Vago_2018, Mathis_2012, Mosheiff_2017}).

In the present work, we take a different statistical viewpoint from that of unbiased estimation using Fisher information: we take a testing approach with a minimax viewpoint. 
Our first argument is that estimation is very complex in fact. For instance, when using Cramér-Rao bound, one has to be aware that it applies only to unbiased estimators  and  the Stein phenomenon shows that biased estimator might sometimes be faster that $I_n(s)^{-1/2}$ 
\cite{Tsybakov}. Minimax theory \cite{Tsybakov} can help to shed more lights on the right order of magnitude of the error by computing the minimal value of $\max_{s}\sqrt{\mathbb{E}(|\hat{s}_n-s|^2)}$.
Our second argument is that pointwise testing (e.g. testing a point like ``$s=0$" for instance) is easier than estimation: imagine an estimator $\hat{s}_n$ of $s$ whose variance depends on $s$ itself. When testing $s=0$, we know the variance of $\hat{s}_n$ under this hypothesis  and we can therefore use this to describe a rejection region for the test. Going further,  the fact that tests are easier to build than estimation, can lead to surprising theoretical effects: depending on the regularity of the $s$ to estimate, it has been shown that typically one makes a minimax estimation error of say $\Delta_n$, with $n$ the number of observations, whereas there exists some test that can detect that $s\neq 0$ as soon as the distance between $s$ and $0$ is larger than $\rho_n$, with $\rho_n$ that is negligible with respect to $\Delta_n$ \cite{Ingster_1993, BHL_2003}. The difference is more than a mere multiplicative constant: testing rates can be faster than estimation rates. Therefore we want to adopt this testing point of view to see if it can improve our understanding of the place cells/grid cells code.

From this testing point of view, we think that
a good encoding system should be able to discriminate quickly between two stimuli or positions $s_1$ and $s_2$, as soon as they are sufficiently apart. In particular, the testing procedure can take into account at which distance $s_1$ is from $s_2$ (information that cannot be taken into account, at least as explicitly,  in an estimation procedure) and one can have a discrimination time between two points that depends on this distance.

Therefore the purpose of the present work is to  study the following three theoretical problems. To make things more concrete, imagine a rat in a maze who should learn a certain behavior in position $s_1$ and another one in position $s_2$. We give for each problem the ``rat" interpretation with respect to this situation. 

\bigskip
\hspace*{-0.5cm} \fbox{%
   \begin{minipage}{0.95\textwidth}
\begin{definition}
 Given $n$ neurons obeying a certain stochastic model parametrized by a code $f(s)$ in response to a stimulus/position $s\in \mathbb{S}$ presented to the system for a duration $T$, 
we define the \textit{minimal discrimination time} between two locations $s_1$ and $s_2$ for the code $f$, with precision $\alpha$ (denoted  $T_{min}(f,s_1,s_2,\alpha)$), as the minimal time the output of the $n$ neurons must be observed in order to distinguish $s_1$ from $s_2$ with a probability of error less than $\alpha \in (0,1)$. 
\end{definition}
\end{minipage}
}

\bigskip

\textbf{Problem 1} consists in understanding the behavior of the minimal discrimination time $T_{min}(f,s_1,s_2,\alpha)$. 
If this minimal time is infinite, it expresses in particular the fact that the coding system cannot distinguish $s_1$ from $s_2$. More precisely, this quantity also expresses how the minimal discrimination time decreases when the distance between $s_1$ and $s_2$ increases. Note in particular that the test which is able to achieve this minimal time can depend on the precise knowledge of $s_1$ and $s_2$ because we define this as a test and this will be the case in our solution. From the ``rat" perspective, the discrimination time between the two positions, is a good lower bound for the reaction time of the rat to this given task, because this is only after realising that it is in position $s_1$ or $s_2$ that the rat can proceed to the learned associated behavior.

\bigskip
\hspace*{-0.5cm}\fbox{%
   \begin{minipage}{0.95\textwidth}
\begin{definition}
We equip $\mathbb{S}$ with a certain metric $d$. We define the \textit{minimax discrimination time} of a family of codes $\mathcal{F}$ at distance $\rho$ by \begin{equation}\label{minimax}
T(\mathcal{F}, \rho,\alpha)=\inf_{f\in \mathcal{F}} \sup_{s_1,s_2 \in \mathbb{S}: \ d(s_1,s_2)\geq \rho} T_{min}(f,s_1,s_2,\alpha).
\end{equation}
This quantity can be seen as the  rate at which the best code $f$ in a certain family $\mathcal{F}$ (for instance place cells or grid cells) can discriminate all stimuli at distance $\rho$ or more.
\end{definition}
\end{minipage}
}

\bigskip
\textbf{Problem 2} consists in computing upper and lower bounds for this minimax discrimination time. In particular, this rate is a function  of $n$ and $\rho$ and it is not clear whether the best code $f$ depends on $\rho$ or not.
From the ``rat" perspective, it implies that the important parameter for the distinction between $s_1$ and $s_2$ is  the distance: if the brain ``uses'' the best code in the family $\mathcal{F}$ at distance at least $\rho$", then as soon as $d(s_1,s_2)\geq \rho$, one can guarantee a reaction time that is at least $T(\mathcal{F}, \rho,\alpha)$. From a modeling point of view, this raises a good question: why would the brain ``use the best code in the family $\mathcal{F}$ at  distance at least $\rho$" ?. The important part in the previous sentence is in particular: why would the brain focus on one particular $\rho$? This leads to the third problem: adaptation.

\bigskip
\hspace*{-0.5cm}\fbox{%
   \begin{minipage}{0.95\textwidth}
\begin{definition}
A code $f\in \mathcal{F}$ is said to be adaptive if it achieves the rate defined in \eqref{minimax}, up to multiplicative constants, in a given range of values for $\rho$, that is,
\begin{equation}\label{adapt}
\forall \rho \ \mbox{in a given range},  \sup_{s_1,s_2 \in \mathbb{S} : \ d(s_1,s_2)\geq \rho} T_{min}(f,s_1,s_2,\alpha)\simeq T(\mathcal{F}, \rho,\alpha).
\end{equation}
\end{definition}
\end{minipage}
}

\bigskip

\textbf{Problem 3} consists in finding an adaptive code and the corresponding range of $\rho$ for different family of codes. 
Here the word \textit{adaptation} is meant in the sense of theoretical statistics/minimax theory \cite{Ingster_1993,BHL_2003,Tsybakov}. 
From the "rat" perspective, adaptation (in the previous statistical sense) is fundamental. We can indeed pinpoint two scenarios about the learning. \textit{In Scenario 1}, before even learning, the system (place cells or grid cells) can achieve for many $\rho$ the best discrimination time (adaptive code). The only thing that the rat has to learn is the specific response of the cells to position $s_1$ and $s_2$ to perform the ``best" discrimination test and decide what is the correct behavior. This scenario has the advantage of minimal training time: if a new  couple $(s_1,s_2)$ is presented, the learning time should be very quick. In \textit{Scenario 2}, the system is not adaptive and then each time a new couple $(s_1,s_2)$ at a different distance is proposed, the rat either is stuck with  a suboptimal code leading to a suboptimal discrimation time, or it has to learn a new representation/new code at the same time as the new couple to react faster. Scenario 2 is of course less "adaptive" because the rat would take a longer time to react/learn.

Solving these three theoretical problems should help us to distinguish between several types of codes and in particular between place cells and grid cells. The minimal discrimination time computed in Problem 1 is an idealized reference of a certain reaction time, which depends on the coding system. It is already slightly more complex than simple estimation accuracy since it encompasses the idea that if the stimuli are very different, the reaction time should be faster. Second, the minimax time in Problem 2 should in particular tell us which system seems to be more competitive than the other. Finally, the adaptation viewpoint in Problem 3 should give a more specific viewpoint on the practical use of each system to see if it is possible to discriminate stimuli at different scales with the same code.

To push the mathematics as far as possible, we use a very simple stochastic model: the spike trains of the $n$ neurons are homogeneous Poisson processes (i.e. constant firing rates in time, that only depends on the position/stimulus $s$) with coding performed only via their respective firing rates. We also idealize the rate code $f$ into simple step functions with only two values and used the circle as the stimulus/position set (which is consistent with, for example, a 1D circular maze or with the direction of movement \cite{Georgopoulos}).

We have been able to completely answer all three problems for the place cells code, for which adaptation is possible. All three problems are also solved for grid cell codes  when the number of cells per module and scales are fixed, but  minimax rate and adaptation for general grid cells code, where even the number of modules is let free is still an open problem. In particular, we have shown that the minimal distance $\rho$ that can be detected by $n$ place cells is of order $1/n$, a distance that is much smaller than 
the rate $n^{-1/2}$ obtained via the classical use of Fisher information
(see \cite{Brunel_1998}).  It also appears that grid cells have much better resolution than place cells, up to $2^{-n}$, which is consistent with the bounds given by \cite{Fiete_2011}. Grid cells may also be faster for discrimination than place cells, for a fixed $\rho$, if the system is well calibrated. However, it does not seem clear that there is an adaptive code for general grid cells, and in this sense, place cells might have an advantage in terms of reaction time for sufficiently distant stimuli/positions.

In Section 2, we give the stochastic model, the main notation and compute the minimal discrimination time (Problem 1) in a very general setting, as well as a lower bound in $2^{-n}$ on   the smallest distance $\rho$ that can be discriminated whatever the code. In Section 3, we study more deeply the place cells code, compute minimax rates (Problem 2) in $\lfloor n \rho\rfloor ^{-1}$ and prove that even random codes are adaptive  in this setting (Problem 3).  In Section 4, we study grid cells code and prove that it can reach the resolution $2^{-n}$ and that another grid code can also achieve the rate $\lfloor n /\log_2(1/\rho) \rfloor ^{-1}$. Numerical illustration is provided in Section 5. Conclusion, Discussion and Perspectives are given in Section 6. Auxiliary results are postponed to Section 7.

\section{Discrimination time}

\subsection{Model and notation}
Since we are interested in grid cells that have a periodic feature, it is simpler to represent stimuli as a circle than as an interval. 

We consider stimulus/position $s$ that belong to $\mathbb{S}^1=[0,1)$, that is considered in a periodic way, i.e. $0\equiv 1$. Equivalently, we can represent this set of stimuli/position as a circle. 

\begin{figure}
    \centering
    \includegraphics[width=0.95\textwidth]{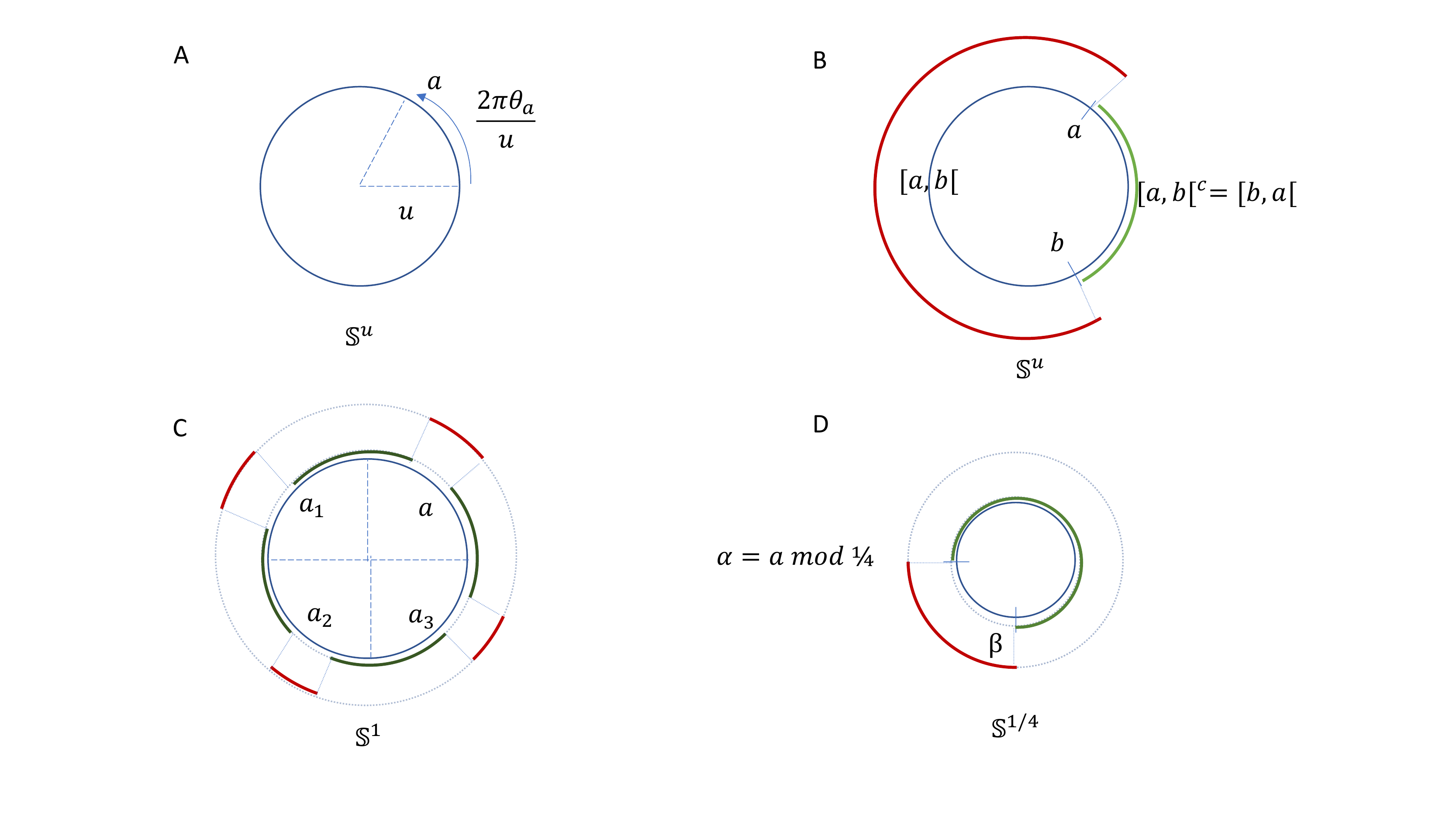}
    \caption{Visual representation of the main notions. In A,  $\mathbb{S}^u$, the circle of radius $u$ with a point $a$ and its corresponding argument $\theta_a\in [0,u)$ 
    In B, the visual representation of the intervals. In red and green, the visual representation of the function $s \mapsto g(s)=\mu_{\llb a,b\llb }(s)+1_{\llb a,b\llb ^c}(s)$, with the value $\mu$ in red and $1$ in green. In C and D, visual representation of the action of $mod$. In these pictures, $\alpha=a~\mathrm{mod}~1/4=a_1~\mathrm{mod}~ 1/4=a_2~\mathrm{mod}~1/4 = a_3~\mathrm{mod}~1/4$. Also in D, the representation of the function $s\mapsto g_{1/4}(s)=\mu 1_{\llb \alpha,\beta\llb }(s)+1_{\llb \alpha,\beta\llb^c}(s)$ with the same color code as in B. In C, the representation of the periodic function $s \mapsto g(s)=g_{1/4}(\modsu) = \mu 1_{\llb\alpha,\beta\llb}(s~\mathrm{mod}~1/4)+1_{\llb \alpha,\beta\llb ^c}(s~\mathrm{mod}~1/4)$}
    \label{fig:rep}
\end{figure}

The stimulus is encoded by $n$ neurons, which emit spikes as \textit{independent homogeneous Poisson processes}. By homogeneous Poisson process, we mean that if the stimulus $s$ remains constant through time, then the firing rate of the Poisson process is also constant through time. More specifically, for a given stimulus $s$ in $\mathbb{S}^1$,  each neuron $i$ has a firing rate $f_i(s)$, which only depends on $s$, as long as the stimulus $s$ is presented.

In the literature, to model cells with receptor fields, these functions $f_i(s)$ are most of the time centered around a certain favorite stimulus for which the value is very large, whereas it returns to some very small rate when it is far from this favorite stimulus.

To simplify the mathematical computations, we decide to use piecewise constant functions to model $f_i$. More precisely, these functions only take two values $\mu$ and $1$ with $\mu$ much larger than 1. We say that neuron $i$ responds to stimulus $s$ or is activated by $s$ when $f_i(s)=\mu$. We denote the code $f=(f_1,...,f_n)$, the vector of the firing rates and $I^f_s$ is the set of neurons responding to $s$ that is
\[I^f_s=\{i \in [n]~:~ f_i(s)=\mu\},\]
with $[n]$ the short notation for $\{1,...,n\}$. 
In what follows, for any subset $B\subseteq [n]$, we denote $|B|$ the cardinality of the set B.

To define more precise the piecewise constant functions $f_i$, especially for the grid cells code, we need to introduce further mathematical notation. Let  $\mathbb{S}^u$ be the circle of radius $u$, which can be put in one-to-one map with the segment $[0,u)$, through the argument $a\mapsto \theta_a$ as seen in Figure \ref{fig:rep}.A. More mathematically, for a given $a \in \mathbb{S}^u$, $\theta_a$ is the only value in $[0,u)$ such that $a$ corresponds to the point $u(\cos(2\pi\theta_a/u),\sin(2\pi\theta_a/u))$. The distance we use on $\mathbb{S}^u$ is the geodesic distance on the circle divided by $2\pi$. This can also be viewed as $d_u(a,b)=\min(|\theta_a-\theta_b|,(u-|\theta_a-\theta_b|))$. At most, it is $u/2$. When $u=1$, we write $d$ instead of $d_1$ for short. Observe that, in $\mathbb{S}^1$, the largest distance one can have is $1/2$. The interval $\llb a, b\llb$ is defined as the set of all points $s$ such that $\theta_s\in [\theta_a,\theta_b)$ if $\theta_a<\theta_b$, and the set of all points the $s$ such that $\theta_s\notin [\theta_b,\theta_a)$ if $\theta_b<\theta_a$
with the convention that $\llb a, a\llb=\emptyset$ is the empty set. Note that the complementary of $\llb a, b\llb$, satisfies $\llb a, b\llb^c=\llb b, a \llb$. 
See also Figure \ref{fig:rep}.B.

The first code that we are interested in, corresponds (via this piecewise constant simplification) to classical neurons having a certain receptor field, or to place cells, as mentioned in the introduction.

\paragraph{Place cells code $\mathcal{P}$.}
A code $f$ is a place cells code if and only if for all $i$ in $[n]=\{1,...,n\}$, there exist $a_i,b_i$ in $\mathbb{S}^1$ such that for all $s$ in $\mathbb{S}^1$,
\begin{equation}
\label{place-def}
 f_i(s)= \mu 1_{\llb a_i,b_i\llb }(s)+1_{\llb a_i,b_i\llb^c}(s).    
\end{equation}
 With this definition, we can assimilate the receptor field of neuron $i$ with the interval $\llb a_i,b_i\llb$. A typical representation of this place cells code can be seen Figure \ref{fig:rep}.B

\bigskip

Grid cells have a periodic structure. Therefore, we need to define properly periodic functions as well. To do so, we use the modulus operation. It is defined for $x\in\bR$ and $u\in\bR_{>0}$ by $\modxu=x-\lfloor x/u\rfloor u \in [0,u) 
$, the remainder of the euclidean division of $x$ by $u,$ where $\lfloor x/u\rfloor$ is the largest integer less than or equal to $x/u$. Informally, a function $g$ of period $u$ on $\mathbb{S}^1$ is a function that only depends on the value of the stimulus $s \ \mathrm{mod}\ u$, i.e. $g(s)=g(s \ \mathrm{mod}\ u)$ for all $s\in \mathbb{S}^1$. 
Let us define module on circles more formally:
for radii $u,v\in\bR_{>0}$ and any $s\in\bS^v$, we denote $t=\modsu$ the point in $\bS^u$ such that $\theta^u_t=\theta^v_s \ \text{mod} \ u.$ 
    To give an intuition of what the modulus operation does, we put an example in Figure \ref{fig:rep}C and D.
In this sense, we can formally define periodic functions $s \mapsto g(s)$  on $\mathbb{S}^1$ with period $u$ by saying that $g(s)=g_u(\modsu)$, for some function $g_u$ on $\mathbb{S}^u$.

However, if we cannot divide $\mathbb{S}^1$ in an integer number of intervals of 
length $u$, the periodic functions are not completely well defined. This remark leads to the restriction on the $\lambda_i$'s in the following definition of the grid cells code.

\paragraph{Grid cells code $\mathcal{G}((n_i,\lambda_i)_{i=1,...,m})$.} 
\cite{stensola_et_al_12} showed experimentally that grid cells are grouped by modules. A given module is dedicated to a certain scale (or spatial periodicity of the firing pattern). Once the scale is fixed the exact localisation of the centers of the grid does not seem to show any particular structure. But the progression of the scales seems to be done in quantized manner. We model this as follows.
A code $f$ is a grid cells code with $m$ modules of cardinals $(n_i)_{i=1,...,m}$ and scales $(\lambda_i)_{i=1,...,m}$ if and only if the $n$ neurons are grouped into $m$ modules $M_1,\ldots, M_m$ so that there is  $n_i:=|M_i|$ cells per module $i$ with $\sum_{i=1}^m n_i=n$ and all these $n_i$ cells have a periodic code of period $\lambda_i$.
More precisely, for a given neuron $j\in M_i$, there exists $a_{i,j}$ and $b_{i,j}$ in $\mathbb{S}^{\lambda_i}$, the circle of radius $\lambda_i$, such that for all $s\in \mathbb{S}^1$
\begin{equation}
\label{grid-def}
f_{i,j}(s)=
\mu 1_{\llb a_{i,j},b_{i,j}\llb}(\modsi)+1_{\llb a_{i,j},b_{i,j}\llb^c}(\modsi)
\end{equation}
To ensure coherence of the respective periods in each module, we assume that $1:=\lambda_1>\lambda_2>\ldots >\lambda_m$ are real positive numbers satisfying the following relations: $\lambda^{-1}_i\in\bN_{>0}$ and $\lambda_{i}\lambda^{-1}_{i+1}\in\bN_{>0}$. In particular $\lambda_{i+1}\leq \lambda_i/2$ for all module $i$.
A typical representation of this grid cells code can be seen Figure \ref{fig:rep}.C.

\paragraph{Unspecified grid cells code $\bar{\cG}$.} If we want to consider a grid cell class for which the number of modules, scales and number of cells per module are not specified, one can consider
\[\bar{\cG}=\cup_{m=1,...,n}\cup_{n_1+...+n_m=n} \cup_{\lambda_1>...>\lambda_m }\mathcal{G}((n_i,\lambda_i)_{i=1,...,m}), \]
where the last union is taken only on the $\lambda_i$'s such that $ \lambda_i\lambda_{i+1}^{-1} \in \bN_>0$ for $i<m$.

\medskip

Note that grid cells code with only one module are place cells code and this also justifies why it was easier to consider stimulus on the circle in the first place.

Some of the results we are going to prove also hold for more general binary codes.

\paragraph{General binary code}
A code $f$ is a general binary code if and only if for all $i$ in $[n]=\{1,...,n\}$, $s\mapsto f_i(s)$ is a piecewise constant function on $\mathbb{S}^1$ taking only two possible values, $\mu$ and $1$. 
Place cells code and grid cells code are just particular cases of general binary codes. In a general binary code, the set of stimulus to which neuron $i$ responds forms a borelian of $\mathbb{S}^1$. 

Notice that the choice of $1$ as the smallest of the two possible rates for the code is to simplify computations, but one can always transform the data to be in this case. Indeed, the time-changing theorem \cite{brown_2002} allows us to change time in order to fix the smallest rate at 1.

\subsection{The statistical testing problem}

A stimulus $s$ is applied for a time $T$ and it results from the stimulus $s$,  $n$ spike trains for the $n$ different neurons, i.e. $N^1,...,N^n$, the $n$ independent Poisson processes on $[0,T]$. We consider that the individual (or the agent or the brain) has only these spike trains as source of information on $s$ and that it tries to use the best statistical tool available, based on $N^1,...,N^n$. This philosophy gives us an ideal bound on the performance that the brain can do with one or another encoding system. It is typically used by \cite{Brunel_1998} to say that the inverse of the Fisher Information gives  the smallest variance of an estimator of the stimulus (Cramer-Rao bound) and that maximizing the Fisher Information gives the best code. We want to basically adopt the same point of view but for testing instead of estimating.

In the testing problem, given $s_1$ and $s_2$ two possible values for $s$, the individual has to guess that $s=s_1$ or $s=s_2$ based solely on $N^1,...,N^n$. Mathematically, it means that this guess is  a test
$\Phi=\Phi(N^1,...,N^n)$ that can only take two values  $s_1$ or  $s_2$.
The individual can make two mistakes $\mathbb{P}_{s_1}(\Phi=s_2)$ that is the probability that the guess is $s_2$ whereas the applied stimulus is $s_1$, and reciprocally $\mathbb{P}_{s_2}(\Phi=s_1)$, that is the probability that the guess is $s_1$ whereas the applied stimulus is $s_2$.
There are varieties of possible tests and we are interested only by the ones, which have (up to multiplicative constants) the smallest possible error, that is we want to find $\Phi$ such that it achieves
\begin{equation}\label{proba_error}
p_e(s_1,s_2)=\min_{\Phi \mathrm{~test~of ~}s_1 \mathrm{~vs ~} s_2} p_{e,\Phi}\quad\mbox{ with }\quad p_{e,\Phi}= \max\left[\mathbb{P}_{s_1}(\Phi=s_2),\mathbb{P}_{s_2}(\Phi=s_1)\right].
\end{equation}

\paragraph{Order of magnitude of $p_{e}$}

For $s_1,s_2 \in \bS^{1}$, we denote 
$\Delta_{s_1,s_2}^f=\max(|I^f_{s_1}\setminus I^f_{s_2}|, |I^f_{s_2}\setminus I^f_{s_1}|).$

\begin{proposition}
\label{bounds_pe}
For all $T>0$, $\mu>1$ and $s_1,s_2 \in \bS^1$
\[
 \max\left\{\frac{\exp\left[-T\tilde{C}_{\mu}\Delta_{s_1,s_2}^f\right]}{4}, \frac{1-\sqrt{T\tilde{C}_{\mu}\Delta_{s_1,s_2}^f/2}}{2} \right\} \ \leq \quad p_e(s_1,s_2) \ \leq \ \exp\left[-T  C_{\mu} \Delta_{s_1,s_2}^f\right], \]
with $C_{\mu}=\frac{(\mu-1)^2}{4}\min\left\{\frac{1}{2\mu}, \frac{3}{5+\mu}\right\}$ and $\tilde{C}_{\mu}=(\mu-1)\log(\mu).$ 

\end{proposition}
\begin{proof}
We set $I^f_1=I^f_{s_1}$ (resp. $I^f_2=I^f_{s_2}$), the set of neurons activated by $s_1$ (resp. $s_2$). We also denote $\bP_{1}$ (resp. $\bP_2$) the distribution of $N^1,...,N^n$ and $\bE_1$ (resp. $\bE_2$) the corresponding expectation, when the applied stimulus is $s_1$ (resp. $s_2$). 
First notice that the Kullback-Leibler distance between both distribution is $K(\bP_{1},\bP_{2})=-\bE_1(\log(\frac{d\bP_2}{d\bP_1}))$, where $\frac{d\bP_2}{d\bP_1}$ the Radon-Nikodym derivative of $\bP_2$ with respect to $\bP_1$. One can check that
\begin{align*}
\frac{d\bP_2}{d\bP_1}&=\exp\left\{-(\mu-1)T (|I^f_{2}|-|I^f_1|)+\log(\mu)\left(\sum_{i\in I^f_2}N^i_T-\sum_{i\in I^f_1}N^i_T\right)\right\}\\
&=\exp\left\{-(\mu-1)T (I^f_{2}\setminus I^f_1|-|I^f_1\setminus I^f_2|)+\log(\mu)\left(\sum_{i\in I^f_2}N^i_T-\sum_{i\in I^f_1}N^i_T\right)\right\},
\end{align*}
and
\[
\bE_1\left(\sum_{i\in I^f_2}N^i_T-\sum_{i\in I^f_1}N^i_T\right)=T(|I_{2}^f\setminus I^f_1|-\mu|I^f_1\setminus I_2^f|),
\]
so that
\begin{align*}
K(\bP_{1},\bP_{2})&=T(\mu-1)(|I^f_{2}\setminus I^f_1|-|I^f_1\setminus I^f_2|)-\log(\mu)T(|I^f_{2}\setminus I^f_1|-\mu|I^f_1\setminus I^f_2|)\\
&= T(\mu-1-\log(\mu))|I^f_2\setminus I^f_1|+T(\mu\log(\mu)-\mu+1)|I^f_1\setminus I^f_2|.
\end{align*}
As a consequence, it follows that
\begin{align*}
K(\bP_{1},\bP_{2})&\leq  T \Delta_{s_1,s_2}^f(\mu-1-\log(\mu)+\mu\log(\mu)-\mu+1)\\
&= \tilde{C}_{\mu}T \Delta_{s_1,s_2}^f.
\end{align*}
The lower bound on $p_e$ follows from Theorem 2.2 of \cite{Tsybakov}.

Now let us assume that $\Delta^f_{s_1,s_2}=|I^f_1\setminus I^f_2|$ (if this is not the case we exchange $s_1$ and $s_2$). It is now sufficient to consider the test
$\Psi$  defined by
\begin{equation}
\label{def:optimal_test_1}
\Psi(N^1,\ldots, N^n)=
\begin{cases}
s_1, \ \text{if} \ Z:= \sum_{i\in I^f_1\setminus I^f_2}N^i_T > |I^f_1\setminus I^f_2|T(\mu+1)/2\\
s_2, \ \text{otherwise}
\end{cases}.
\end{equation}
Indeed, 
\[
\bP_{2}(\Psi=s_1)=\bP_2(Z>|I^f_1\setminus I^f_2|T(\mu+1)/2)=\bP_2(Z>T|I^f_1\setminus I^f_2|(1+(\mu-1)/2)).
\]
Under $\bP_2$, we have that $Z\sim\text{Poi}(T|I^f_1\setminus I^f_2|)$. Thus, by applying inequality \eqref{Prop:Poi_upper} of Lemma \ref{Prop:Poi_Upper_Lower} with $\theta=T|I^f_1\setminus I^f_2|$ and $x=(\mu-1)/2$, it follows that
\begin{align*}
\bP_2(\Psi=s_1)\leq \exp\left\{-\frac{3T|I^f_1\setminus I^f_2|(\mu-1)^2}{4(5+\mu)}\right\}.
\end{align*}
Under $\bP_1$, we know that $Z\sim\text{Poi}(T\mu|I^f_1\setminus I^f_2|)$, so that by applying inequality \eqref{Prop:Poi_lower} of Lemma \ref{Prop:Poi_Upper_Lower} with $\theta=T\mu|I_1\setminus I_2|$ and $x=(\mu-1)/2)T|I_1\setminus I_2|$, we deduce that
\begin{align*}
\bP_{1}(\Psi=s_2)&=\bP_1(Z\leq T\mu|I^f_1\setminus I^f_2|-(\mu-1)/2)T|I^f_1\setminus I^f_2|)\leq \exp\left\{-\frac{(\mu-1)^2T|I^f_1\setminus I^f_2|}{8\mu}\right\},
\end{align*}
which concludes the proof.
\end{proof}

\subsection{Minimal discrimination time}
From Proposition \ref{bounds_pe} we have that for a given admissible error level $\alpha\in (0,1)$, if 
 \begin{enumerate}
    \item If $T\tilde{C}_{\mu}\Delta^f_{s_1,s_2}<\log(1/4\alpha)=\log(1/\alpha)-\log(4)$, then $p_{e}>\alpha$.
    \item If $T  C_{\mu} \Delta^f_{s_1,s_2}\geq \log(1/\alpha)$, then $p_{e}\leq \alpha.$ 
\end{enumerate}

Hence the minimal discrimination time stated in Problem 1, $T_{min}(f,s_1,s_2,\alpha)$, is of the order of $1/\Delta^f_{s_1,s_2} $, up to positive multiplicative constants in $\alpha$ and $\mu$ (see also Section 5 for numerical verification). This turns the other statistical problems into combinatorial problems. In particular, discrimination is not possible if and only if  $\Delta^f_{s_1,s_2}=0$, that is $I^f_{s_1}=I^f_{s_2}$. 

The behavior of the quantities in $\alpha$ and $\mu$ are quite intuitive. If the level $\alpha$ tends to 0, $T_{min}$ tends to infinity. If $\mu$ tends to infinity, $T_{min}$ tends to 0.  Considering that $\mu$ is some biological parameter that is fixed, as well as the admissible level $\alpha$ of reliability of the system, we are now focusing on the behavior of $1/\Delta^f_{s_1,s_2}$ and from now on, we denote with a slight abuse of language
\begin{equation}\label{Tmin}
 T_{min}(f,s_1,s_2)=\frac1{\Delta^f_{s_1,s_2}}.   
\end{equation}

Let us now introduce the minimal time for which one can distinguish any pair of stimuli which are at least $\rho$ apart with $\rho\in (0,1/2]$:
\begin{equation}\label{Trho}
T(f,\rho)=\max_{s_1,s_2\in \bS^{1}: d(s_1,s_2)\geq \rho }T_{min}(f,s_1,s_2).
\end{equation}
Clearly, the function $\rho\mapsto T(f,\rho)$ is non increasing (the larger the distance between two stimuli the smaller the time one needs to observe the activity of the network to distinguish one from the other).

Before speaking of minimax codes (Problem 2), let us derive an absolute lower bound on the range that can be detected by a general binary code $f$.


\begin{proposition}
\label{prop:distinction_general_code}
Let $f$ be a general binary code. Then, for any $0\leq \rho<1/2^{n}$, $T(f,\rho)=\infty$.
\end{proposition}

This means that whatever the code, there exists always two stimuli at distance less than $2^{-n}$ that cannot be discriminated, whatever the observation time. 

\begin{proof}
For  $A\subseteq [n]$, let $S_A$ be the set of all $s\in\bS$ such that $I^f_s=A$. Notice that  for each $i\in [n]$ 
\[
S_A=\left(\cap_{i\in A}f^{-1}_i(\{\mu\})\right)\cap\left(\cap_{i\in A^c}f^{-1}_i(\{1\})\right),
\]
is a borelian of $\bS^1$ for each $A\subseteq [n]$.

Notice also that if $s\in S_A\cap S_B$, then $A=I^f_s=B$. Moreover, $\cup_{A\subseteq [n]}S_A=\bS^1$. Hence, $\{S_A\}_{A\subseteq [n]}$ forms a partition of $\bS^1$.
Let us denote $\text{Leb}_*$ the pushforward (a measure on $(\bS^{1},\cS)$) of the Lebesgue on $[0,1)$ induced by the map $[0,1)\in \theta\mapsto (\cos(2\pi\theta),\sin(2\pi\theta))\in\bS$. By  definition of $\text{Leb}_*$ it holds that $\text{Leb}_*(\llb a, b \llb)=2\pi(\theta_b-\theta_a)$ and  $\text{Leb}_*(\bS^1)=2\pi.$

Since $\{S_A\}_{A\subseteq [n]}$ forms a partition of $\bS^{1}$ there  exists $A\subseteq [n]$ such that $\text{Leb}_*(S_A)\geq 2\pi/2^n$. 
Suppose that $\text{diam}(S_A)=:\sup_{s_1,s_2\in S_A}d(s_1,s_2)\leq \rho$. If this was true, then the set $S_A$ would be contained in an interval of $\text{Leb}_*$-measure $2\pi\rho $ so that $2\pi/2^n\leq \text{Leb}_*(S_A)\leq 2\pi\rho$, implying that $\rho\geq 1/2^{n}$, a contradiction.

Therefore, if $\rho<1/2^{n}$, then we must have that 
\[
\text{diam}(S_A)>\rho, 
\]
implying that we can find $s_1,s_2\in S_A$ such that $d(s_1,s_2)>\rho$. In this case, $s_1$ and $s_2$ cannot be distinguished and $T_{min}(f,\rho)=\infty$. 
\end{proof}

\section{Results for place cells code}
\label{sec:codes_type_1}

In this section, we focus our analysis on the class of place cells code $\cP$ defined in \eqref{place-def}. Our first result is a lower bound for 
\begin{equation}
\label{minimax_place_cells_up_to_constants}
T(\cP,\rho)=\inf_{f\in\cP}T(f,\rho),    
\end{equation}
which is (up to multiplicative constants in $\alpha$ and $\mu$) the minimax discrimination time for the class of place cells code defined in \eqref{minimax}.

\subsection{Lower bound on the minimax discrimination time}

\begin{proposition}
\label{prop_lower_bound_place_class}
For any $0\leq \rho\leq 1/2$, we have that 
\[
T(\cP,\rho)\geq \frac{1}{\lfloor n/\lfloor (2\rho)^{-1}\rfloor \rfloor}\geq \frac{1}{\lfloor 4n\rho\rfloor}.
\]
In particular, for any $0\leq \rho\leq (2(n+1))^{-1}$, we have that $T(\cP,\rho)=\infty.$
\end{proposition}
\begin{proof}
Let $0\leq \rho\leq 1/2$ and take $L=\lfloor 1/2\rho \rfloor$ so that $L\leq 1/2\rho<L+1$ and hence $L>1/2\rho-1$. 
Set $\theta_0=0$ and $1\leq k\leq L$, write $\theta_k=k\rho$ and $\theta_{-k}=1-k\rho$. For each $f\in\cP$, denote 
$m^f_{\rho}=\min\{\Delta^f_{\rho_k,\rho_{k-1}}: -L+1\leq k\leq L \}$, where $\rho_k\in \bS^{1}$ is identified with its argument $\theta_k$. 

Note that \[I^f_{\rho_k}\setminus I^f_{\rho_{k-1}}=\{i\in [n]:a_i\in \rrb \rho_{k-1},\rho_k\rrb \ \text{and} \ b_i\notin \rrb\rho_{k-1},\rho_k\rrb \}\subset \{i\in [n]:a_i\in \rrb\rho_{k-1},\rho_k\rrb\}\]
and 
\[I^f_{\rho_{k-1}}\setminus I^f_{\rho_{k}}=\{i\in [n]:b_i\in \rrb \rho_{k-1},\rho_k\rrb \ \text{and} \ a_i\notin \rrb \rho_{k-1},\rho_k\rrb\}\subset
\{i\in [n]:b_i\in \rrb \rho_{k-1},\rho_k\rrb\}.\]
Hence, it follows that 
\[\sum_{k=-L+1}^{L}|I^f_{\rho_k}\setminus I^f_{\rho_{k-1}}|\leq \sum_{k=-L+1}^{L}|\{i\in [n]:a_i\in \rrb \rho_{k-1},\rho_k\rrb\}|\leq n,\]
and
\[\sum_{k=-L+1}^{L}|I^f_{\rho_{k-1}}\setminus I^f_{\rho_{k}}|\leq \sum_{k=-L+1}^{L}|\{i\in [n]:b_i\in \rrb \rho_{k-1},\rho_k\rrb\}|\leq n.\]
Therefore, it follows that
\[
\sum_{k=-L+1}^{L}\Delta^f_{\rho_k,\rho_{k-1}}=\sum_{k=-L+1}^{L}\max\left\{|I^f_{\rho_k}\setminus I^f_{\rho_{k-1}}|, |I^f_{\rho_{k-1}}\setminus I^f_{\rho_{k}}|\right\}\leq 2n,
\]
implying that $2Lm^f_{\rho}\leq 2n$, that is, $m^f_{\rho}\leq n/L.$
Since $m^f_{\rho}$ is an integer, we must have $m^f_{\rho}\leq \lfloor n/L \rfloor=\lfloor n/\lfloor (2\rho)^{-1}\rfloor \rfloor$
so that
\[
\min_{s_1,s_2\in\bS^{1}:d(s_1,s_2)\geq \rho}\Delta^f_{s_1,s_2}\leq m^f_{\rho}\leq \lfloor n/\lfloor (2\rho)^{-1}\rfloor \rfloor.
\]
Since $f\in\cP$ is arbitrary, the above inequality ensures that
\[
T^{-1}(\cP,\rho)=\sup_{f\in\cP}\min_{s_1,s_2\in\bS^{1}:d(s_1,s_2)\geq \rho}\Delta^f_{s_1,s_2}\leq \lfloor n/\lfloor (2\rho)^{-1}\rfloor \rfloor.
\]
Since $(2\rho)^{-1}\geq 1$ which implies that $\lfloor (2\rho)^{-1}\rfloor\geq 1$ we have that
$4\lfloor (2\rho)^{-1}\rfloor\geq 2(\lfloor (2\rho)^{-1}\rfloor+1)\geq \rho^{-1}$, so that 
the first part of the proof follows from inequality above.

To conclude the proof, observe that 
if $\rho\leq (2(n+1))^{-1}$ then $\lfloor (2\rho)^{-1}\rfloor\geq n+1>n$, implying that $\left\lfloor n/\lfloor (2\rho)^{-1}\rfloor\right\rfloor=0$
and in this case $T(\cP,\rho)=\infty.$
\end{proof}

To better understand the lower bound provided by Proposition \ref{prop_lower_bound_place_class}, let us adopt an asymptotic point of view in which the number of observed neurons $n\to\infty$ and $\rho=\rho_n$ is a function of $n$. For example, when $\rho_n$ is a constant function of $n$ the lower bound decreases as $1/n$, whereas in the regime $\rho_n\to 0$ such that $n\rho_n\to \infty$ the lower bound behaves as $(2n\rho_n)^{-1}$. Moreover, as long as $\rho_n\leq (2(n+1))^{-1}\approx (2n)^{-1}$, whatever the code $f\in\cP$, there exists always two stimulus at distance $\rho_n$ that cannot be distinguished, whatever the observation time.

\subsection{Upper bounds and minimax codes}

Next we obtain an upper bound for $T(\cP,\rho)$ matching the lower bound above, up to multiplicative constants. To that end, we  study the behavior of $T(f,\rho)$ for some examples of place cell codes and use the fact that $T(\cP,\rho)\leq T(f,\rho)$ for any $f\in\cP$.    

\begin{example}[1-Uniform code]
For each $i\in\{0,1,\ldots,n\}$, let $p_i$ be the point in $\bS^{1}$ associated with $\theta_{i}=\frac{i}{n}$. The 1-uniform code is defined as $f^{n,n}=(h_1,\ldots,h_{n})$ (the superscript $n,n$ stands for $n$ neurons divided into $n$ groups of size 1)  where each $h_i$ is given by \eqref{place-def} with $a_i=p_{i-1}$ and $b_{i}=p_{i}$. 
It is straightforward to check that $T(f^{1,1},\rho)=\infty$ for all $0\leq \rho< 1/2$ and $T(f^{1,1},1/2)=1$. Let us analyse the case $n\geq 2$.
In this case, 
\[
T(f^{n,n},\rho)=
\begin{cases}
\infty, \ \text{if} \ 0\leq \rho<1/n,\\
1, \ \text{otherwise}
\end{cases}.
\]
To check that, suppose first that $\rho<1/n$. In this case, take $\theta_{\rho}$ such that $\rho\leq \theta_{\rho}<1/n$ , denote $p_{\rho}$ the point in $\bS^{1}$ associated to $\theta_{\rho}$ and observe that
$
p_{\rho}\in \llb p_0, p_1\llb,
$
$I^{f^{n,n}}_{p_{\rho}}=I^{f^{n,n}}_{p_{0}}$ and $d(p_{\rho},p_0)=\theta_{\rho}\geq\rho$. Hence,  $T(f^{n,n},\rho)=\infty.$
Now, suppose $1/n\leq \rho\leq 1/2$. In this case, it is not difficult to check that $I^{f^{n,n}}_{s_1}\neq I^{f^{n,n}}_{s_2}$ for all $s_1,s_2\in\bS^{1}$ such that $d(s_1,s_2)\geq\rho$  which implies that
\[
\Delta^{f^{n,n}}_{s_1,s_2}\geq 1.
\]
But since $|I^{f^{n,n}}_{s}|=1$ for all $s\in\bS^{1}$, we deduce that
\[
\min_{s_1,s_2\in\bS^1:d(s_1,s_2)\geq \rho}\Delta^{f^{n,n}}_{s_1,s_2}= 1,
\]
and the result follows.
\end{example}

\begin{example}[d-Uniform code]
\label{ex:d_unifor_code}
For each $k\in\{0,1,\ldots,d\}$, let $p_k$ denote the point in $\bS^{1}$
defined by $\theta_k=\frac{k}{d}$ and let $h_k$ be given by \eqref{place-def} with $a_k=p_{k-1}$ and $b_{k}=p_{k}$.
Also, write $L=\lfloor n/d\rfloor$. The d-uniform code is defined as $f^{n,d}=(f_1,\ldots,f_{n})$ (the superscript $n,d$ stands for $n$ neurons divided into $d$ groups) where $f_i=h_k$ for $(k-1)L+1\leq i\leq kL$ with $k\in\{1,\ldots,d-1\}$, and  $f_i=h_{d}$ for $n\geq i\geq (d-1)L+1$.
One can check that for any $n\geq 2$,
\[
T(f^{n,d},\rho)=
\begin{cases}
\infty, \ \text{if} \ 0\leq \rho<1/d,\\
1/\lfloor n/d\rfloor, \ \text{otherwise}
\end{cases}.
\]
To verify that, consider first the case that $\rho<1/d$. In this case, take $\theta_{\rho}$ such that $\rho\leq \theta_{\rho}<1/d$, denote $p_{\rho}$ the point in $\bS^{1}$ given by $\theta_{\rho}$ and observe that
$
p_{\rho}\in \llb p_0, p_1\llb,
$
$I^{f^{n,d}}_{p_{\rho}}=I^{f^{n,d}}_{p_{0}}$ and $d(p_{\rho},p_0)=\theta_{\rho}\geq\rho$. Thus, we have that $T(f^{n,d},\rho)=\infty.$
Suppose now that $1/d\leq \rho\leq 1/2$. In this case, one can check that 
$I^{f^{n,d}}_{s_1}\cap I^{f^{n,d}}_{s_2}=\emptyset$ for all $s_1,s_2\in\bS^{1}$ such that $d(s_1,s_2)\geq\rho$ so that 
\[
\Delta^{f^{n,d}}_{s_1,s_2}=\max\left\{|I^{f^{n,d}}_{s_1}\setminus I^{f^{n,d}}_{s_2}|, |I^{f^{n,d}}_{s_2}\setminus I^{f^{n,d}}_{s_1}|
\right\}=\max\left\{|I^{f^{n,d}}_{s_1}|, |I^{f^{n,d}}_{s_2}|\right\}\geq L.
\]
Since the lower bound L is attained for some pair $s_1,s_2\in \bS^{1}$ (for example, take $s_1$ and $s_2$ given by $\theta_{s_1}=0$ and $\theta_{s_2}=1/2$ respectively), the results follows.
\end{example}

This codes allows us to prove the following result.

\begin{corollary}\label{minimax_place}
For all $\rho \in (1/n,1/2]$, let $d=\lceil 1/\rho \rceil$. Then the d-Uniform code satisfies
\[\frac{1}{\lfloor4n\rho\rfloor}\leq T(\mathcal{P},\rho)\leq T(f^{n,d},\rho) \leq \frac{1}{\lfloor n/d\rfloor} \leq \frac{1}{\lfloor3n\rho/2\rfloor}\]
and is therefore minimax up to multiplicative constant.
\end{corollary}

\subsection{A particular adaptive code}
Now that we have found a minimax code for given $\rho$ up to multiplicative constant, one could wonder if there exists a given place cells code, which for every $\rho$ in a given range of values, achieves the minimax rate up to a constant. We refer to this problem as Problem 3, adaptation to the distance. The following code is a particular example of such a code.
\begin{example}[An adaptive place cells code]
\label{ex_adaptive_code}
 Let $g=(g_1,\ldots,g_n)$ be the code in $\cP$ with $f_i$ defined by \eqref{place-def} where the points $a_i,b_i\in \bS^{1}$ are associated to $\theta_{i}=i/2n$ and $\theta_{i}+1/2$ respectively. A visualisation is given in Figure \ref{fig:adapt}.
 \begin{figure}
 \centering
\includegraphics[width=6cm]{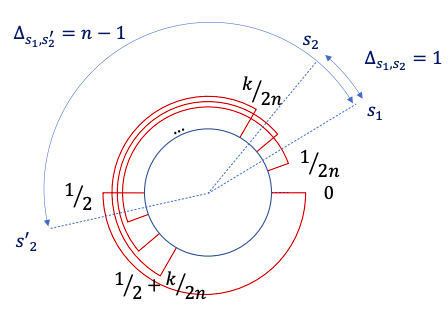}
\caption{Visual representation of Examples 3. For an easier visualisation, the different half red circles have a different radius. Nevertheless, each of them correspond a certain cell $i$ of the code, and more precisely to the locations $s$ such that $f_i(s)=\mu$. Two couples are considered $(s_1,s_2)$ (standing for two positions that are very close) and $(s_1,s'_2)$ (standing for positions that are very far). We see that $\Delta_{s_1,s'_2}=n-1>>\Delta_{s_1,s_2}=1$.} 
    \label{fig:adapt}
\end{figure}
 For the code $g$, one can show the following.

\begin{proposition} 
\label{T_min_g_code}
For any $n\geq 1$ and $0\leq \rho\leq 1/2$, 
\[
\frac{1}{\lfloor 2n\rho\rfloor}\leq T(g,\rho)\leq \frac{2}{\lfloor 2n\rho\rfloor}.
\]
In particular the $g$ code cannot discriminate at distance $\rho<1/(2n)$ and 
it is adaptive to the distance in the class $\cP$, up to some absolute multiplicative constant, in the range $\rho\in [1/(2n),1/2]$.
\end{proposition}
\begin{proof}
 One can easily check that for any $0\leq s< 1,$
\[
I^g_s=
\begin{cases}
\emptyset, \ \text{if} \ 0\leq \theta_s<1/2n\\
\{1,\ldots, k\}, \ \text{if} \ k/2n\leq \theta_s<(k+1)/2n \ \text{with} \ 1\leq k\leq n\\
\{k+1,\ldots, n\}, \ \text{if} \ (n+k)/2n\leq  \theta_s< (n+k+1)/2n \ \text{with} \ 1\leq k\leq n-1
\end{cases}.
\]
Thus, for any $0\leq \rho<1/2n$, it follows that if one takes $s_1,s_2\in\bS^{1}$ with $\theta_{s_1}=0$ and $\rho \leq \theta_{s_2}<1/2n$, then $I^g_{s_1}=I^g_{s_2}=\emptyset$ and $d(s_2,s_1)=|\theta_{s_2}-\theta_{s_1}|=\theta_{s_2}\geq \rho$, implying that $T_{min}(g,\rho)=\infty$, and the result holds since $\lfloor 2n\rho\rfloor=0.$ 

Now, take $1/2n<\rho\leq 1/2$ and $s_1\in \bS^{1}$. Suppose that $0\leq \theta_{s_1}<1/2$. Let $k=\lfloor 2n\rho\rfloor$ and $q\in \{0,\ldots, n-1\}$ be such that $q/2n\leq \theta_{s_1}<(q+1)/2n$. 

Let us consider first the case $\theta_{s_1}+\rho<(n+1)/2n.$
In this case, for all $\theta_{s_1}+\rho\leq \theta_{s_2}<(n+1)/2n$, we have that $I^{g}_{s_1}\setminus I^{g}_{s_2}=\emptyset$ and $\{q+1,\ldots, q+k\}\subseteq I^{g}_{s_2}\setminus I^{g}_{s_1}$ so that $\Delta^g_{s_1,s_2}=\max\left\{|I^{g}_{s_1}\setminus I^{g}_{s_2}|, |I^{g}_{s_2}\setminus I^{g}_{s_1}|\right\}\geq q$ (and the equality can be achieved). 
For $(n+1)/2n\leq \theta_{s_2}\leq \theta_{s_1}+1/2$, one can check that
$
I^{g}_{s_2}\setminus I^{g}_{s_1}=\{q+1,\ldots, n\}
$
so that 
$\Delta^g_{s_1,s_2}=\max\left\{|I^{g}_{s_1}\setminus I^{g}_{s_2}|,|I^{g}_{s_2}\setminus I^{g}_{s_1}|\right\}\geq n-q\geq k.$

Hence, we have shown that for any $s_1\in\bS^{1}$ such that $0\leq \theta_{s_1}+\rho<1/2$,
\begin{equation}
\label{ineq_1_proof_lower_bound_g_code_place_cells}
\min_{s_2\in \bS^{1}:d(s_2,s_1)\geq \rho,\theta_{s_2}\leq \theta_{s_1}+1/2 }\Delta^g_{s_1,s_2}\geq k.
\end{equation}

Let us now consider the case $\theta_{s_1}+\rho\geq (n+1)/2n.$
In this case, one can check 
$ n+1\leq \lfloor 2n(\theta_{s_1}+\rho)\rfloor\leq n+q$. Let $1\leq r\leq q$ be such that $n+r=\lfloor 2n(\theta_{s_1}+\rho)\rfloor$. Since $2n(\theta_{s_1}+\rho)\geq q+k$, it follows that $n+r\geq q+k$, implying that $k\leq (n-q)+r$.
Now, note that for $\theta_{s_1}+\rho\leq \theta_{s_2}\leq  \theta_{s_1}+1/2$, we have that
$I^g_{s_2}\setminus I^g_{s_1}=\{q+1,\ldots, n\}$ and $I^g_{s_1}\setminus I^g_{s_2}=\{1,\ldots, \ell\}$ with $r\leq \ell\leq q$. Hence, we have that
\[
\Delta^g_{s_1,s_2}=\max\left\{|I^{g}_{s_1}\setminus I^{g}_{s_2}|, |I^{g}_{s_2}\setminus I^{g}_{s_1}|\right\}\geq \max\left\{(n-q), r\right\} \geq \frac{1}{2}(n-q+r)\geq k/2,
\]
so that for any $s_1\in\bS^{1}$ such that $\theta_{s_1}+\rho\geq (n+1)/2n$,
\begin{equation}
\label{ineq_2_proof_lower_bound_g_code_place_cells}
\min_{s_2\in \bS^{1}:d(s_2,s_1)\geq \rho,\theta_{s_2}\leq \theta_{s_1}+1/2 }\Delta^g_{s_1,s_2}\geq k/2.
\end{equation}
Therefore, combining inequalities \eqref{ineq_1_proof_lower_bound_g_code_place_cells} and \eqref{ineq_2_proof_lower_bound_g_code_place_cells} we deduce that for any $s_1\in \bS^{1}$ with $0\leq \theta_{s_1}<1/2$,
\[
\min_{s_2\in \bS^{1}:d(s_2,s_1)\geq \rho,\theta_{s_2}\leq \theta_{s_1}+1/2 }\Delta^g_{s_1,s_2}\geq k/2.
\]
By similar arguments, one can also show that for any $s_1\in \bS^{1}$ with $0\leq \theta_{s_1}<1/2$
\[
\min_{s_2\in \bS^{1}:d(s_2,s_1)\geq \rho,\theta_{s_2}\geq \theta_{s_1}+1/2 }\Delta^g_{s_1,s_2}\geq k/2=\lfloor 2n\rho \rfloor/2.
\]
The case $1/2\leq \theta_{s_1}\leq 1$ follows along the same lines of reasoning by observing that if $\theta_{\tilde{s}_1}=\theta_{s_1}-1/2$ then $0\leq \theta_{\tilde{s}_1}<1/2$ and  $\theta_{\tilde{s}_1}+1/2=\theta_{s_1}$. 

To conclude the proof, observe that 
\[
\min_{s_1\in \bS^{1}}\min_{s_2\in \bS^{1}:d(s_2,s_1)\geq \rho}\Delta^g_{s_1,s_2}\leq \Delta^g_{s_0,a_{k}}=|I^g_{k/2n}|=k=\lfloor 2n\rho \rfloor,
\]
where $s_0$ is the point $\bS^{1}$ defined by $\theta_{0}=0$
and the result follows.
\end{proof}
\end{example}


\subsection{Random codes are adaptive as well}
In the next example, we show that even random codes are adaptive (in the range $\rho\geq n^{-1/2}$, up to multiplicative constants) with large probability.  

\begin{example}[Random code]
\label{ex_random_code}
Let $A_1,\ldots, A_n$ and $B_1,\ldots, B_n$ be independent and uniformly distributed points on $\bS^{1}$. 
The random code $f^{r}$ is defined as $f^{r}=(f^{r}_1,\ldots,f^{r}_n)$, where for each $i\in [n]$, $f_i$ is of the form \eqref{place-def} with $a_i=A_i$ and $b_i=B_i$. For the random code $f^r$, we can prove the following result.
\begin{proposition}
\label{random_code_adaptive_place_cells}
There exist constants $K_1,K_2>0$ such that for each $n\geq K_1$ the following property holds:  for each $x\in\bR_{>0}$ and $0\leq \rho\leq 1/2$,
\begin{equation}
\label{upper-bound-Tmin-random-code}
T(f^r,\rho)\leq \frac{1}{\left\lceil \frac{n}{2}\left(\rho-K_2n^{-1/2}(1+x^{1/2}(1+(x/n)^{1/2}))\right) \right\rceil},
\end{equation}
 with probability at least $1-\exp{(-x)}$. 
 In particular, for any $\delta\in (0,1/4)$ and each $n$ satisfying 
 \[n\geq \max\{K_1,(1-4\delta)^2/(4K^2_2\delta^2(1+x^{1/2})^2)\}:=K_3,\]we have the following: if $1/2\geq \rho\geq n^{-1/2}K_2(1+x^{1/2}+x)/(1-4\delta)$ then
\begin{equation}
\label{Tmin-random-code-adaptive}
T(f^r,\rho)\leq\frac{1}{\lceil \delta n\rho \rceil},
\end{equation}
with probability at most $1-\exp(-x).$
\end{proposition}

\begin{proof}
First observe that to prove inequality in \eqref{upper-bound-Tmin-random-code}, it suffices to show that
\begin{equation}
\label{ineq_0_proof_random_code_place_cells}
\min_{s_1,s_2\in\bS^1: d(s_1,s_2)\geq\rho}\Delta^{f^r}_{s_1,s_2}\geq \frac{n}{2}\left(\rho-K_2n^{-1/2}(1+x^{1/2}(1+(x/n)^{1/2}))\right)
\end{equation}
with probability at least $1-\exp{(-x)}$.
To that end, first observe that
\begin{align*}
|I^{f^r}_{s_1}\setminus I^{f^r}_{s_2}|+|I^{f^r}_{s_2}\setminus I^{f^r}_{s_1}|& =\sum_{i=1}^n {\bf 1}_{s_1\in \llb A_i, B_i\llb, s_2\in \llb B_i, A_i\llb }+{\bf 1}_{s_2\in \llb A_i, B_i\llb, s_1\in \llb B_i, A_i\llb}\\
& =  n\left(1-\frac{1}{n}\sum_{i=1}^{n}{\bf 1}_{s_1,s_2\in \llb A_i, B_i\llb}+{\bf 1}_{s_1,s_2\in \llb B_i, A_i\llb}\right)   
\end{align*}
which implies that
\begin{align*}
\Delta^{f^r}_{s_1,s_2}& \geq \frac{1}{2}\left(|I^{f^r}_{s_1}\setminus I^{f^r}_{s_2}|+|I^{f^r}_{s_2}\setminus I^{f^r}_{s_1}|\right)\\
& = \frac{n}{2}\left(1-\frac{1}{n}\sum_{i=1}^{n}{\bf 1}_{s_1,s_2\in \llb A_i, B_i\llb}+{\bf 1}_{s_1,s_2\in \llb B_i, A_i\llb}\right).
\end{align*}
Write $\xi_i=(A_i,B_i)\in \bS^{1}\times\bS^{1}$ for $i\in [n]$, let $C_{s_1,s_2}=\{(a,b)\in \bS^{1}\times\bS^{1}: s_1,s_2\in \llb a,b\rrb  \ \text{or} \ s_1,s_2\in \llb b,a,\llb\}$ for each $s_1,s_2\in\bS^{1}$. 
With this notation, we deduce from the previous inequality that
\[
\min_{s_1,s_2\in\bS^1: d(s_1,s_2)\geq\rho}\Delta^{f^r}_{s_1,s_2}\geq \frac{n}{2}\left(1-\max_{s_1,s_2\in\bS^1:d(s_1,s_2)\geq \rho}\frac{1}{n}\sum_{i=1}^n{\bf 1}_{C_{s_1,s_2}}(\xi_i)\right).
\]
Now, by Lemma \ref{prob_of_unif_interval_having_2_specific_points},
\[
\bE\left({\bf 1}_{C_{s_1,s_2}}(\xi_i)\right)=2\bP(s_1,s_2\in \llb A_i,B_i\llb )=2\left(\frac{1}{2}-|\theta_{s_2}-\theta_{s_1}|\left(1-|\theta_{s_2}-\theta_{s_1}|\right)\right).
\]
Hence, it follows that for $s_1,s_2\in\bS^{1}$ such that $d(s_1,s_2)\geq \rho$ (that is $\rho \leq |\theta_{s_2}-\theta_{s_1}|\leq 1-\rho$), we have that
\[
\bE\left(1_{C_{s_1,s_2}}(\xi_i)\right)\leq 1-2\rho(1-\rho)\leq 1-\rho,
\]
where in the last inequality we have used that $0\leq \rho\leq 1/2$.
From the above inequality, we deduce that  
 \[
\max_{s_1,s_2\in\bS^1:d(s_1,s_2)\geq \rho}\frac{1}{n}\sum_{i=1}^n1_{C_{s_1,s_2}}(\xi_i)\leq (1-\rho)+W_{\cC},
\]
where $W_{\cC}$ is the random variable defined as
\[
W_{\cC}=\max_{C\in\cC}\frac{1}{n}\sum_{i=1}^n\left(1_{C}(\xi_i)-\bE\left(1_{C}(\xi_i)\right)\right),
\]
with the set $\cC=\left\{C_{s_1,s_2}:s_1,s_2\in\bS^1:d(s_1,s_2)\geq \rho\right\}$.
Combining the above inequalities, we deduce that
\begin{equation}
\label{ineq_1_proof_random_code_place_cells}
\min_{s_1,s_2\in\bS^1: d(s_1,s_2)\geq\rho}\Delta^{f^r}_{s_1,s_2}\geq \frac{n}{2}\left(\rho-W_{\cC}\right).
\end{equation}
By applying the Bousquet inequality (e.g., see inequality (5.49) on page 170 of \cite{Massart}), we deduce that for any $x\in\bR_{>0}$,
\begin{equation}
\label{ineq_2_proof_random_code_place_cells}
\bP\left(W_{\cC}\leq \bE\left(W_{\cC}\right)+\sqrt{\frac{2x}{n}(1-\rho+\bE\left(W_{\cC}\right))}+\frac{x}{3n}\right)\geq 1-\exp(-x).    
\end{equation}
From now on, we denote $K$ a positive constant which can change from line to line. Let us denote $V(\cC)$ the VC-dimension of the set $\cC$.
By Lemma 6.4 of \cite{Massart}, there exists an absolute constant $K$ such that, for all $n$ for which 
$1-\rho\geq K^2V(\cC)(1-\frac{1}{2}\log(1-\rho))/n$, it holds that
\[
\bE\left(W_{\cC}\right)\leq \frac{K}{2}\sqrt{\frac{(1-\rho)}{n}V(\cC)(1-\frac{1}{2}\log(1-\rho))}\leq \frac{K}{2}\sqrt{\frac{V(\cC)}{n}(1+\frac{1}{2}\log(2))},
\]
where in the last inequality we have used that $0\leq \rho\leq 1/2.$
Combining the above inequality with \eqref{ineq_2_proof_random_code_place_cells} and by using again that $0\leq \rho\leq 1/2$, it follows that for all n for which $n\geq 2K^2V(\cC)(1+\log(2)/2)$ and any $x\in\bR_{>0},$ 
\[
W_{\cC}\leq n^{-1/2}\left(\frac{K}{2}\sqrt{V(\cC)(1+\log(2))}+x^{1/2}\sqrt{2\left(1+\frac{K}{2}\sqrt{V(\cC)(1+\log(2))}\right)}+x/3n^{1/2}\right)
\]
with probability at least $1-\exp(-x)$.
Let us assume that the VC-dimension $V(\cC)$ is bounded by some absolute constant. In this case, it follows from the above inequality and \eqref{ineq_1_proof_random_code_place_cells} that for all $n\geq K^2$ and any $x\in\bR_{>0}$,
\[
\min_{s_1,s_2\in\bS^1: d(s_1,s_2)\geq\rho}\Delta^{f^r}_{s_1,s_2}\geq \frac{n}{2}\left(\rho-n^{-1/2}K\left(1+x^{1/2}\left(1+(x/n)^{1/2}\right)\right)\right)
\]
with probability at least $1-\exp(-x)$, proving  \eqref{ineq_0_proof_random_code_place_cells}.
For the VC-dimension $V(\mathcal{C})$, there are various ways to see that it is finite. One way is to say that $\xi=(A,B)\in C_{s_1,s_2}$ is equivalent to saying that $(\theta_A,\theta_B)$ belongs to the union of 
\begin{itemize}
    \item $[0,\min(\theta_{s_1},\theta_{s_2})]\times(\max(\theta_{s_1},\theta_{s_2}),1]$
    \item $(\max(\theta_{s_1},\theta_{s_2}),1]\times [0,\min(\theta_{s_1},\theta_{s_2})]$
    \item $[\min(\theta_{s_1},\theta_{s_2}),\max(\theta_{s_1},\theta_{s_2}))\times [\min(\theta_{s_1},\theta_{s_2}),\max(\theta_{s_1},\theta_{s_2}))$
    \item $[0,\min(\theta_{s_1},\theta_{s_2})]\times [0,\min(\theta_{s_1},\theta_{s_2})]$
    \item $(\max(\theta_{s_1},\theta_{s_2}),1]\times (\max(\theta_{s_1},\theta_{s_2}),1]$
\end{itemize}
Hence it is included in the union of 5 rectangles. It is well known that the VC dimension of the family of rectangles is 4.  Class $\mathcal{C}$ is included in the 5-fold unions of rectangles, whose VC dimension is bounded by $4*5\log(5)$ up to multiplicative constants,  thanks to \cite{Csikos}. Therefore Class $\mathcal{C}$ is of finite VC-dimension, which concludes the proof. 

Note that with a closer look at all possibilities, one can show that class $\mathcal{C}$ cannot shatter samples of $\xi_i$'s of size 5 and that in fact its VC dimension is therefore 4. However, the proof with all the possibilities is much longer than using the bound given by \cite{Csikos}.

Let us now prove \eqref{Tmin-random-code-adaptive}. Denote $\eta_{n,x}=(1+x^{1/2}+x/n^{1/2})$. Clearly $(1+x^{1/2}+x)\geq \eta_{n,x}\geq (1+x^{1/2})$. Moreover, one can check that if $n\geq K_3$, then $2\delta(K_2\eta_{n,x})/(1-4\delta)\geq n^{-1/2}$ which implies that
$(K_2\eta_{n,x})/(1-4\delta)\geq 2n^{-1/2}+K_2\eta_{n,x}$.     
Hence, if $\rho\geq n^{-1/2}K_2(1+x^{1/2}+x)/(1-4\delta)\geq n^{-1/2}K_2\eta_{n,x}/(1-4\delta)$, then 
\[
n(\rho-K_2\eta_{n,x}/n^{1/2})/2\geq 1
\]
so that
\[
\lfloor n(\rho-K_2\eta_{n,x}/n^{1/2})/2 \rfloor\geq n(\rho-K_2\eta_{n,x}/n^{1/2})/4\geq n\rho \delta, 
\]
and the result follows from \eqref{upper-bound-Tmin-random-code}.

\end{proof}
\end{example}

\hspace*{-0.5cm}\fbox{%
   \begin{minipage}{0.95\textwidth}
\begin{center}   
\textbf{\large Summary of the results on place cells codes}
\end{center}

\begin{enumerate}
    \item 
    For any given place cells code $f$ of size $n$, and any $0\leq \rho\leq 1/(2(n+1))$, there exists a pair $(s_1,s_2)$ at distance at least  $\rho$ that cannot be discriminated. 
    \item For any $\rho\geq 1/(2n)$, the minimax discrimination rate for place cell codes is  of order $1/(n \rho)$.
    \item The code $g$ defined in Example \ref{ex_adaptive_code} is adaptive, meaning that it reaches this rate for all $\rho\geq 1/(2n)$, up to multiplicative constants.
    \item Random codes $f^{r}$ defined in Example~\ref{ex_random_code} are adaptive in the range $\rho \geq c /n^{1/2}$, for some constant $c>0$, up to multiplicative constants, with large probability.
    \end{enumerate}
\end{minipage}%
}

\section{Results for grid cells code}
\label{sec:codes_type_2}
In this section, we discuss some results for the class of grid cells code defined in \eqref{grid-def} $\cG((n_i,\lambda_i)_{i=1,...,m})$   and the more general $\bar{\cG}$. 
In what follows, for $f\in\bar{\cG}$, $s_1,s_2\in\bS^{1}$ and $i\in [m]$, let us denote 
\[
\Delta^{f,i}_{s_1,s_2}=\max\left\{|I^{f,i}_{s_1}\setminus I^{f,i}_{s_2}|, |I^{f,i}_{s_2}\setminus I^{f,i}_{s_1}|\right\}, 
\]
where $I^{f,i}_{s}=\{j\in M_i: f_{ji}(s)=\mu\}$ is the set of cells activated by stimulus $s\in\bS^{1}$ in the $i$-th module. The first result establishes a useful link between grid cells code and place cells code. 
\begin{proposition}
\label{prop-link-grid-place-cells}
For any $f\in\bar{\cG}$ and $s_1,s_2\in\bS^{1}$, the following inequality holds
\begin{equation}
\sum_{i=1}^m\Delta^{f,i}_{s_1,s_2}\geq \Delta^{f}_{s_1,s_2}\geq \frac{1}{2}\sum_{i=1}^m\Delta^{f,i}_{s_1,s_2}.    
\end{equation}
\end{proposition}
\begin{proof}
Suppose that $\Delta^{f}_{s_1,s_2}=|I^f_{s_1}\setminus I^f_{s_2}|$ (if not exchange $s_1$ and $s_2$). On one hand, observe that since the modules are disjoint we have that
\[
|I^f_{s_1}\setminus I^f_{s_2}|=\sum_{i=1}^m |I^{f,i}_{s_1}\setminus I^{f,i}_{s_2}|\leq \sum_{i=1}^m \max\left\{|I^{f,i}_{s_1}\setminus I^{f,i}_{s_2}|, |I^{f,i}_{s_2}\setminus I^{f,i}_{s_1}|\right\} = \sum_{i=1}^m\Delta^{f,i}_{s_1,s_2}. 
\]
On the other hand, using that
\[
|I^f_{s_1}\setminus I^f_{s_2}|+|I^f_{s_2}\setminus I^f_{s_1}|=\sum_{i=1}^m |I^{f,i}_{s_1}\setminus I^{f,i}_{s_2}|+|I^{f,i}_{s_2}\setminus I^{f,i}_{s_1}|\geq \sum_{i=1}^m\max\left\{|I^{f,i}_{s_1}\setminus I^{f,i}_{s_2}|, |I^{f,i}_{s_2}\setminus I^{f,i}_{s_1}|\right\},
\]
we obtain that 
\[
2\Delta^f_{s_1,s_2}\geq |I^f_{s_1}\setminus I^f_{s_2}|+|I^f_{s_2}\setminus I^f_{s_1}|\geq \sum_{i=1}^m \Delta^{f,i}_{s_1,s_2},
\]
and the result follows.
\end{proof}

\subsection{Lower bound on the minimax discrimination time for $\cG((n_i,\lambda_i)_{i=1,...,m})$  }

The following result provides a lower bound on the minimax discrimination time (up to constants depending on $\alpha$ and $\mu$) defined in \eqref{minimax}, when restricted to the class of grid cells code: 
\begin{equation}
\label{minimax_grid_cells_up_to_constants}
T(\cG((n_i,\lambda_i)_{i=1,...,m}),\rho)=\inf_{f\in\cG((n_i,\lambda_i)_{i=1,...,m})}T(f,\rho).  
\end{equation}

\begin{theorem}
\label{thm_lower_bound_grid_class}
If 
$0\leq \rho\leq 1/2$ and $\lambda_1/\lambda_2$ is even or
if $0\leq \rho\leq 1/2-\lambda_2/2$ and $\lambda_1/\lambda_2$ is odd,
we have that
\[T(\cG((n_i,\lambda_i)_{i=1,...,m}),\rho)\geq \frac{1}{\left\lfloor 6 \min_{1\leq k\leq m}\left\{\sum_{i=1}^{k}\frac{n_i}{\lambda_i}\max\left\{\rho, \lambda_{k+1}\right\}\right\}\right\rfloor},
\]
with the convention that $\lambda_{m+1}=0$.
In particular, if we define $j_{\rho}=\max\{1\leq k\leq m: \lambda_k\geq \rho\}$, then $T(\cG((n_i,\lambda_i)_{i=1,...,m}),\rho)=\infty$ whenever one of the following conditions hold: 
\begin{itemize}
    \item either there exists $2\leq k\leq j_\rho$ such that \[(k-1)\lambda_k<\min_{1\leq i<k}\left[\frac{\lambda_i}{6n_i}\right],\]
    \item or $j_\rho\rho < \min_{1\leq i\leq j_\rho}\left[\frac{\lambda_i}{6n_i}\right]$.
\end{itemize}
\end{theorem}

\begin{proof}
We need to compute an upper bound for $\min_{s_1,s_2\in\bS^{1}:d(s_1,s_2)\geq\rho}\Delta^f_{s_1,s_2}$ uniformly on $f\in\cF_2$.
By the definition of $j_{\rho}$, if $j_{\rho}\geq 2$ then we have that
\[
\lambda_k\geq \rho>\lambda_{j_{\rho}+1}, \ \text{for all} \ 2\leq k\leq j_{\rho} \ \text{and} \ \rho\leq \lambda_1/2.
\]
When $j_{\rho}=1$ it holds that $\lambda_{2}<\rho\leq 1/2=\lambda_1/2.$



Let us consider first the case $j_{\rho}\geq 2$. 
In this case, let $2\leq k\leq j_{\rho}$ and take $s_1,s_2\in\bS^{1}$ such that $d(s_1,s_2)=\lambda_k\geq \rho$. By Lemma \ref{mult_distance_equal_mod} it follows that $s_1 \ \mod \ \lambda_i=s_2 \ \mod \ \lambda_i$ for all $k\leq i\leq m$ so that for all $j\in M_i$,
\[
f_{ij}(s_1)=f_{ij}(s_1 \ \mod \ \lambda_i)=f_{ij}(s_2 \ \mod \ \lambda_i)=f_{ij}(s_2),
\]
implying that $\Delta^{f,i}_{s_1,s_2}=0$.

Hence, by Proposition \ref{prop-link-grid-place-cells} we have that
\[
\Delta^{f}_{s_1,s_2}\leq \sum_{i=1}^{k-1}\Delta^{f,i}_{s_1,s_2}.
\]

For each $1\leq i\leq k-1$, let $\rho^{i}_{\ell}$ in $S^{\lambda_i}$ defined by $\theta_{\rho^{i}_{\ell}}=\ell\lambda_k$ with $0\leq \ell\leq D_{i+1:k}$, where $D_{i+1:k}:=\lambda_i/\lambda_k$ is an integer larger than $2$ for all $i<k$. For later use, let us also define $D_{2:1}=1.$ By arguing as in the proof of Proposition \ref{prop_lower_bound_place_class}, one can show that
\[
\sum_{\ell=1}^{D_{i+1:k}}\Delta^{f,i}_{\rho^i_{\ell},\rho^i_{\ell-1}}\leq 2n_i.
\]
Now, take $0\leq \ell\leq D_{2:k}, 1\leq i\leq k-1
$ and $j\in M_i$, and observe that 
$
f_{ij}(\rho^{1}_{\ell})=f_{ij}(\rho^{1}_{\ell} \ \mod \ \lambda_i).$ 

\medskip

\textbf{Claim.} It holds true that $\rho^{1}_{\ell} \ \mod \ \lambda_i=\rho^{i}_{\ell \ \mod \ D_{i+1:k}}$.

\bigskip

\textbf{Proof of the Claim.} Write $\ell=rD_{i+1:k}+\ell \ \mod \ D_{i+1:k}$ for some positive integer $r.$ Since $\lambda_i=D_{i+1:k}\lambda_k$, we deduce that 
\[
\ell \lambda_k=r\lambda_i+(\ell \ \mod \ D_{i+1:k})\lambda_k
\]
with $(\ell \ \mod \ D_{i+1:k})\lambda_k<\lambda_i$. As a consequence,  
we obtain that 
\[
\theta_{\rho^1_{\ell} \ \mod \ \lambda_i}=\theta_{\rho^1_{\ell}} \ \mod \ \lambda_i =\ell \lambda_k \ \mod \ \lambda_i  = (\ell \ \mod \ D_{i+1:k})\lambda_k=\theta_{\rho^i_{\ell \ \mod \ D_{i+1:k}}}, 
\]
implying that $\rho^1_{\ell} \ \mod \ \lambda_i = \rho^{i}_{\ell \ \mod \ D_{i+1:k}}$ (by the injectivity of the function $\bS^{1}\ni a\mapsto \theta_a$.) This concludes the proof of the Claim. 

\bigskip

The claim implies that $f_{ij}(\rho^{1}_{\ell})=f_{ij}(\rho^{i}_{\ell \ \mod \ D_{i+1:k}})$ for each $0\leq \ell\leq D_{2:k}, 1\leq i\leq k-1
$ and $j\in M_i$, so that (notice that $D_{2:k}=D_{2:i}D_{i+1:k}$),
\[
\sum_{\ell=1}^{D_{2:k}}\Delta^{f,i}_{\rho^1_{\ell},\rho^1_{\ell-1}}=D_{2:i}\sum_{\ell=1}^{D_{i+1:k}}\Delta^{f,i}_{\rho^i_{\ell},\rho^i_{\ell-1}}\leq D_{2:i}2n_i.
\]
Therefore, the above inequality implies that 
\[
\sum_{\ell=1}^{D_{2:k}}\sum_{i=1}^{k-1}\Delta^{f,i}_{\rho^1_{\ell},\rho^1_{\ell-1}}\leq 2(D_{2:1}n_1+D_{2:2}n_2+\ldots+D_{2:k-1}n_{k-1})
\]
which, in turn, ensures that 
\begin{align*}
D_{2:k}\min_{s_1,s_2\in\bS^{1}:d(s_1,s_2)=\lambda_k}\left\{\sum_{i=1}^{k-1}\Delta^{f,i}_{s_1,s_2}\right\}&\leq \sum_{\ell=1}^{D_{2:k}}\sum_{i=1}^{k-1}\Delta^{f,i}_{\rho^1_{\ell},\rho^1_{\ell-1}}\\& \leq 2(D_{2:1}n_1+D_{2:2}n_2+\ldots+D_{2:k-1}n_{k-1})    
\end{align*}
 or equivalently, 
 \[
 \min_{s_1,s_2\in\bS^{1}:d(s_1,s_2)=\lambda_k}\left\{\sum_{i=1}^{k-1}\Delta^{f,i}_{s_1,s_2}\right\}\leq \frac{2}{D_{2:k}}(D_{2:1}n_1+D_{2:2}n_2+\ldots+D_{2:k-1}n_{k-1})=2\lambda_{k}\sum_{i=1}^{k-1}\frac{n_i}{\lambda_i}, 
 \]
where in the equality we have used that $D_{2:k}=D_{2:i}D_{i+1:k}$ and $D_{i+1:k}=\lambda_i\lambda^{-1}_k$.

Hence, we have shown that for each $2\leq k\leq j_{\rho}$,
\[
\min_{s_1,s_2\in\bS^{1}:d(s_1,s_2)\geq \rho}\Delta^{f}_{s_1,s_2}\leq 2\lambda_{k}\sum_{i=1}^{k-1}\frac{n_i}{\lambda_i}=\sum_{i=1}^{k-1}\frac{2n_i}{\lambda_i}\max\left\{\lambda_{k},\rho\right\}.
\]

If $\rho>\lambda_{j_{\rho}}/2-\lambda_{j_{\rho}+1}/2$, then
\[\lambda_{j_{\rho}}< 2\rho +\lambda_{j_{\rho}+1}< 3\rho.\]
Hence
\[
\sum_{i=1}^{j_{\rho}-1}\frac{2n_i}{\lambda_i}\max\left\{\lambda_{j_{\rho}},\rho\right\}=\sum_{i=1}^{j_{\rho}-1}\frac{2n_i}{\lambda_i}\lambda_{j_{\rho}}< 3\sum_{i=1}^{j_{\rho}-1}\frac{2n_i}{\lambda_i}\rho
\]
and
for all $k\geq j_{\rho}$,
\[
3\sum_{i=1}^{k}\frac{2n_i}{\lambda_i}\max\left\{\rho, \lambda_{k+1}\right\}\geq 3\sum_{i=1}^{j_{\rho}}\frac{2n_i}{\lambda_i}\rho>3\sum_{i=1}^{j_{\rho}-1}\frac{2n_i}{\lambda_i}\rho,
\]
so that
\begin{multline*}
\min_{s_1,s_2\in\bS^{1}:d(s_1,s_2)\geq \rho}\Delta^{f}_{s_1,s_2}\leq \min_{1\leq k\leq j_{\rho}-1}\left\{\sum_{i=1}^{k}\frac{2n_i}{\lambda_i}\max\left\{\lambda_{k+1}, \rho\right\}\right\}\\
=\min\left\{\min_{1\leq k\leq j_{\rho}-1}\left\{\sum_{i=1}^{k}\frac{2n_i}{\lambda_i}\max\left\{\lambda_{k+1}, \rho\right\}\right\},\min_{j_{\rho}\leq k\leq m}\left\{3\sum_{i=1}^{k}\frac{2n_i}{\lambda_i}\max\left\{\lambda_{k+1}, \rho\right\}\right\}\right\}.
\end{multline*}
Therefore, we deduce that when $j_{\rho}\geq 2$ and $\rho>\lambda_{j_{\rho}}/2-\lambda_{j_{\rho}+1}/2$, it holds that
\begin{equation}
\label{ineq_1_upper_bound}
\min_{s_1,s_2\in\bS^{1}:d(s_1,s_2)\geq \rho}\Delta^{f}_{s_1,s_2}\leq 
3\min_{1\leq k\leq m}\left\{\sum_{i=1}^{k}\frac{2n_i}{\lambda_i}\max\left\{\lambda_{k+1}, \rho\right\}\right\}.
\end{equation}


Now, suppose that $\rho\leq\lambda_{j_{\rho}}/2-\lambda_{j_{\rho}+1}/2$. 
Let $r=\lceil\rho/\lambda_{j_{\rho}+1}\rceil$. If $D_{j_{\rho}+1:j_{\rho}+1}$ is even then $r < D_{j_{\rho}+1:j_{\rho}+1}/2$ because $r<\lambda_{j_{\rho}}/2$. If $D_{j_{\rho}+1:j_{\rho}+1}$ is odd, then  $\lambda_{j_{\rho}}/2-\lambda_{j_{\rho}+1}/2= \lambda_{j_{\rho}+1}(D_{j_{\rho}+1:j_{\rho}+1}-1)/2$ is a multiple of $\lambda_{j_{\rho}+1}$. So in both cases, $r <  D_{j_{\rho}+1:j_{\rho}+1}/2$.
Therefore  $L=\lfloor D_{j_{\rho}+1:j_{\rho}+1}/r\rfloor$ is always larger than 2. 

In module $j_\rho$, we construct $L+1$ different points $\rho_\ell^{j_\rho}$ of argument $\ell r \lambda_{j_{\rho}+1}$, for $\ell=0,...,L$ which leads to $L$ distinct pairs at distance $\bar\rho=r \lambda_{j_{\rho}+1}$.

By periodicity, we deploy these points on module $1$, there are $D_{2:j_\rho}L$ of those. Taking their modulo, they correspond on module $i=1,...,m$  to $D_{i+1:j_\rho}L$ pairs at distance $\bar\rho$. Each of these pairs are at distance a multiple of $\lambda_k$ for $k>j_\rho$ and therefore cannot be detected by the modules for $k>j_\rho$.

Therefore a similar argument as above lead us to
\[\min_{s_1,s_2 \in \bS^1:d(s_1,s_2)=\bar\rho} \Delta^f_{s_1,s_2} \leq \frac{2}{D_{2:j_\rho}L} (D_{2:1}n_1+...+D_{2:j_\rho}n_{j_\rho}).\]
Since $L\geq 2$, $L\geq 2/3 (D_{j_{\rho}+1:j_{\rho}+1}/r)$, hence
\[\min_{s_1,s_2 \in \bS^1:d(s_1,s_2)=\bar\rho} \Delta^f_{s_1,s_2} \leq 3 r \sum_{i=1}^{j_\rho} \frac{n_i}{\lambda_i}\lambda_{j_\rho+1}.\]
But $r\leq 2 \rho/\lambda_{j_{\rho}+1}$. Hence
\[\min_{s_1,s_2 \in \bS^1:d(s_1,s_2)=\bar\rho} \Delta^f_{s_1,s_2} \leq 3 \sum_{i=1}^{j_\rho} \frac{2n_i}{\lambda_i}\rho\]
Also, as before
\[3 \sum_{i=1}^{j_\rho} \frac{2n_i}{\lambda_i}\rho \leq 3 \sum_{i=1}^{k} \frac{2n_i}{\lambda_i}\max\left\{\lambda_{k+1}, \rho\right\},\]
for all $k\geq j_\rho$.
Finally since $\bar{\rho}\geq \rho$, we get that as long as $j_\rho\geq 2$, \eqref{ineq_1_upper_bound} holds.


The last case to treat is  when $j_{\rho}=1$. If $\rho\leq \lambda_1/2-\lambda_2/2$, we do as above with $r=\lceil \rho/\lambda_2\rceil$ and $\bar\rho=r\lambda_2\leq 1/2$, constructing pairs that can be distinguished only by module 1 and the same bound applies.

In the case $\rho\geq\lambda_1/2-\lambda_2/2 $ and $\lambda_1/\lambda_2$ is even, we can always take $s_1=0$ and $s_2=1/2$. Since $\lambda_1/\lambda_2$ is even, $d(s_1,s_2)$ is a multiple of $\lambda_2$ and therefore of all the other $\lambda_i$ for $i\geq 2$. Hence it can only be detected by the first module and
\[\min_{s_1,s_2\in\bS^{1}:d(s_1,s_2)\geq \rho}\Delta^{f}_{s_1,s_2}\leq n_1 \leq 3 n_1 \rho\leq 3 \sum_{i=1}^{k} \frac{2n_i}{\lambda_i}\max\left\{\lambda_{k+1}, \rho\right\},\]
for all $k\geq 1$ and we conclude as for the other cases.

To conclude, let us just remark that if there exists $2\leq k\leq j_\rho$ such that \[(k-1)\lambda_k<\min_{i<k}\left[\frac{\lambda_i}{6n_i}\right]=\frac{1}{6\max_{1\leq i<k}\{n_i/\lambda_i\}},\]
then 
$\sum_{i=1}^{k-1} \frac{n_i}{\lambda_i} \lambda_k \leq (k-1)\lambda_k\max_{1\leq i<k}\{\frac{n_i}{\lambda_i}\}<1/6,$ so that the integer part in the lower bound is null. The same phenomenon appears at $k=j_\rho$ if  $j_\rho\rho < \min_{i\leq j_\rho}\left[\frac{\lambda_i}{6n_i}\right]$.


\end{proof}

To understand better the bound given by Theorem \ref{thm_lower_bound_grid_class}, note that $1/ \lfloor 2n_i\rho/\lambda_i\rfloor$ is in fact the rate for the place cells code in module $i$. So, roughly up to multiplicative constant,   to distinguish at distance $\rho$, the code needs to be such that the modules of the  grid cells have to be coherent (i.e. the period $k\lambda_{k+1}$ needs to be detected by at least one of the modules before, for all $k$ such that $\lambda_{k+1}\geq \rho$) and that $j_\rho \rho$ also needs to be in the detection range of at least one of modules $i\leq j_\rho$. It is not clear whether the factor $k$ should be present or not.
Also, for the rate itself note that $\sum_{i=1}^k \frac{n_i}{\lambda_i} \lambda_{k+1}$ is not necessarily monotonous in $k$. Hence we could be in a situation where even if we are interested by small $\rho$ that cannot be detected by the first module, the rate of detection of $\lambda_2$ by the first module acts as a threshold that impacts the rates at $\rho$.

\subsection{Upper bounds and a particular adaptive code for $\cG((n_i,\lambda_i)_{i=1,...,m})$}

\begin{example}(An adaptive grid cell code)
\label{ex-adpative-code-grid-cell}
 This code is made of adaptive place cells code as in Example 3 for each of the modules. More precisely, the module $i$ is the set of neurons $M_i=\{n_{i-1}+1,\ldots, n_{i-1}+n_i\}$ and $f_{n_i+j,i}$ is associated with the points $a_{ij}, b_{ij}\in \bS^{\lambda_i}$ corresponding to the angles $\theta_{a_{ij}}=\frac{j\lambda_i}{2n_i}$ and $\theta_{b_{ij}}=\frac{(j+n_i)\lambda_i}{2n_i}$. Let us denote this code by $g_{gc}.$
\end{example} 
For the code $g_{gc}$, one can prove the following result.
\begin{proposition} 
\label{T_min_g_code_grid_cells}
For any $0\leq \rho\leq 1/2$,
\[
T(g_{gc},\rho)\leq 
\frac{4}{\min_{1\leq k\leq m}\sum_{j=1}^{k}\left\lfloor \frac{n_j}{\lambda_j} \max\left\{\lambda_{k+1}, \rho\right\} \right\rfloor},
\]
with the convention $\lambda_{m+1}=0$.
In particular, if we define $j_{\rho}=\max\{1\leq k\leq m: \lambda_k\geq \rho\}$, then for all $\rho$ large enough such that
\begin{itemize}
    \item for all $k\leq j_\rho$, there exists $j<k$ such that $ \lambda_k\geq \frac{\lambda_j}{n_j}$
    \item and there exists $j\leq j_\rho$ with $\rho\geq \frac{\lambda_j}{n_j}$, we have that
\end{itemize}
\begin{multline*}
	\frac{1}{6\min_{1\leq k\leq m}\sum_{j=1}^{k} \frac{n_j}{\lambda_j} \max\left\{\lambda_{k+1}, \rho\right\}}\leq T(\cG((n_i,\lambda_i)_{i=1,...,m}),\rho)\\
	\leq T(g_{gc},\rho)\leq 
\frac{16 \log_2(1/\rho)}{\min_{1\leq k\leq m}\sum_{j=1}^{k} \frac{n_j}{\lambda_j} \max\left\{\lambda_{k+1}, \rho\right\}},
\end{multline*}
meaning that $g_{gc}$ is adaptive in this range of $\rho$ up to a factor $\log_2(1/\rho)$.
\end{proposition}
\begin{proof}

Take a pair $s_1,s_2\in \bS^1$ such that $d(s_1,s_2)\geq \rho$ and define 
\[
j_{s_1,s_2}=\max\left\{i\in [n]: \frac{\lambda_i}{2}\geq d(s_1,s_2) \right\}.
\]
 By using Lemma \ref{equal_distance_mod}, we then conclude that for $1\leq j\leq j_{s_1,s_2}$,
 \[
 d(s_1 \ \mod \ \lambda_j, s_2 \ \mod \ \lambda_j)=d(s_1, s_2)\geq \rho.
 \]
 Hence, proceeding as in the proof of Proposition \ref{T_min_g_code}, we have then that 
 \begin{align*}
 \sum_{i=1}^m \Delta^{g_{gc},i}_{s_1,s_2}&\geq
 \sum_{i=1}^{j_{s_1,s_2}} \Delta^{g_{gc},i}_{s_1,s_2}\\
 &\geq \sum_{i=1}^{j_{s_1,s_2}} \min_{u_1,u_2\in\bS ^{\lambda_i}:d_{\lambda_i}(u_1,u_2)\geq \rho}\Delta^{g_{gc},i}_{u_1,u_2}
 \\
 &\geq 
 \frac{1}{2}\sum_{j=1}^{j_{s_1,s_2}}\left\lfloor \frac{2n_j}{\lambda_j}\rho \right\rfloor =\frac{1}{2}\sum_{j=1}^{j_{s_1,s_2}}\left\lfloor \frac{2n_j}{\lambda_j} \max\left\{\frac{\lambda_{j_{s_1,s_2}+1}}{2},\rho\right\} \right\rfloor.
 \end{align*}
 Combining the above inequality and Proposition \ref{prop-link-grid-place-cells}, we deduce that
 \begin{align*}
\min_{s_1,s_2\in\bS^{1}:d(s_1,s_2)\geq \rho}\Delta^{g_{gc}}_{s_1,s_2}& \geq \frac{1}{2}\min_{s_1,s_2\in\bS^{1}:d(s_1,s_2)\geq \rho}
\sum_{i=1}^m \Delta^{g_{gc},i}_{s_1,s_2}\\
&\geq \frac{1}{4} 
\min_{1\leq k\leq m}\sum_{j=1}^{k}\left\lfloor \frac{2n_j}{\lambda_j} \max\left\{\frac{\lambda_{k+1}}{2}, \rho\right\} \right\rfloor.\\
&\geq \frac{1}{4} 
\min_{1\leq k\leq m}
 \sum_{j=1}^{k} \left\lfloor \frac{n_j}{\lambda_j} \max\left\{\lambda_{k+1}, \rho\right\} \right\rfloor.
  \end{align*}
\end{proof}

To conclude the proof we need to show that
\begin{equation}
\label{ineq_1_proof_upper_bound_g_code_grid_cells}
\min_{1\leq k\leq m}\sum_{j=1}^{k}\left\lfloor \frac{n_j}{\lambda_j} \max\left\{\lambda_{k+1}, \rho\right\} \right\rfloor\geq 4\log(\rho^{-1})\min_{1\leq k\leq m}\sum_{j=1}^{k}\frac{n_j}{\lambda_j} \max\left\{\lambda_{k+1}, \rho\right\}.
\end{equation}

To that end, first observe that
\begin{equation}
\label{eq_1_proof_upper_bound_g_code_grid_cells}
\min_{1\leq k\leq m}\sum_{j=1}^{k}\left\lfloor \frac{n_j}{\lambda_j} \max\left\{\lambda_{k+1}, \rho\right\} \right\rfloor=\min\left\{\min_{2\leq k\leq j_{\rho}}\sum_{j=1}^{k-1}\left\lfloor \frac{n_j}{\lambda_j}\lambda_{k} \right\rfloor,\sum_{j=1}^{j_{\rho}}\left\lfloor \frac{n_j}{\lambda_j}\rho \right\rfloor\right\},
\end{equation}
and also that for each $k\leq j_{\rho}$ 
\begin{align*}
\sum_{j=1}^{k-1}\left\lfloor \frac{n_j}{\lambda_j}\lambda_{k} \right\rfloor &\geq \max_{1\leq j<k}\left\lfloor \frac{n_j}{\lambda_j}\lambda_{k} \right\rfloor\\
&\geq \frac{1}{2}\max_{1\leq j<k} \frac{n_j}{\lambda_j}\lambda_{k}\geq \frac{1}{2k}\sum_{j=1}^{k-1}\frac{n_j}{\lambda_j}\lambda_{k}.
\end{align*}
Since $\rho\leq \lambda_{j_{\rho}}\leq 2^{-(j_{\rho}-1)}$ it follows that $2\log_2(\rho^{-1})\geq \log(\rho^{-1})+1\geq j_{\rho}\geq k$, so that from the above inequality we obtain that
\[
\sum_{j=1}^{k-1}\left\lfloor \frac{n_j}{\lambda_j}\lambda_{k} \right\rfloor\geq \frac{1}{4\log_2(\rho^{-1})}\sum_{j=1}^{k-1}\frac{n_j}{\lambda_j}\lambda_{k}.
\]
By similar arguments, one can also show that
\[
\sum_{j=1}^{j_{\rho}}\left\lfloor \frac{n_j}{\lambda_j}\rho \right\rfloor\geq \frac{1}{4\log_2(\rho^{-1})}\sum_{j=1}^{j_{\rho}}\frac{n_j}{\lambda_j}\rho.
\]
Combining the last two inequalities above with \eqref{eq_1_proof_upper_bound_g_code_grid_cells}, we then deduce that
\begin{align*}
\min_{1\leq k\leq m}\sum_{j=1}^{k}\left\lfloor \frac{n_j}{\lambda_j} \max\left\{\lambda_{k+1}, \rho\right\} \right\rfloor &\geq \frac{1}{4\log_2(\rho^{-1})}\min\left\{\min_{2\leq k\leq j_{\rho}}\sum_{j=1}^{k-1} \frac{n_j}{\lambda_j}\lambda_{k},\sum_{j=1}^{j_{\rho}}\frac{n_j}{\lambda_j}\rho \right\}\\
&=\frac{1}{4\log_2(\rho^{-1})}\min_{1\leq k\leq m}\sum_{j=1}^{k}\frac{n_j}{\lambda_j} \max\left\{\lambda_{k+1}, \rho\right\},
\end{align*}
proving \eqref{ineq_1_proof_upper_bound_g_code_grid_cells}, and the result follows.

\subsection{Random codes are also adaptive in $\cG((n_i,\lambda_i)_{i=1,...,m})$}

\begin{example}[Random code for grid cells]
\label{ex_random_code_grid_cekks}
Again we use random codes on each of the modules.
More precisely, for each $1\leq i\leq m$, let $M_i=\{n_{i-1}+1,\ldots, n_{i-1}+n_i\}$ denote the set of neurons in module $i$.
Let $(A_{ij})_{j\in M_i}$ and $(B_{ij})_{j\in M_i}$ be independent and uniformly distributed points on $\bS^{\lambda_i}$. 
The random code $f^{r}_{gc}$ for the grid cells is defined as $f^{r}_{gc}=(f^r_{gc,11},\ldots, f^r_{gc,ij},\ldots,f^r_{gc,Mn})$, where $f^r_{gc,ij}$ is given by \eqref{grid-def} with $a_{ij}=A_{ij}$ and $b_{ij}=B_{ij}$ for each $i\in [m]$ and $j\in M_i$. 

\begin{proposition}
There exist constants $K_1,K_2>0$ such that if $\min_{1\leq i\leq m}\{n_i\}\geq K_1$ the following property holds:  for  $x\in\bR_{>0}$ and $0\leq \rho\leq 1/2$,
\(T(f^r_{gc},\rho)\)  is upper bound by the inverse of the ceiling function of
\[
 \min_{1\leq k\leq m}\displaystyle\sum_{i=1}^k \frac{n_i}{4}\left(\frac{\max\left\{\rho, \frac{\lambda_{k+1}}{2}\right\}}{\lambda_i}-\frac{K_2}{n_i^{1/2}}\left[1+(x+\log(m))^{1/2}\left[1+\left(\frac{(x+\log(m))}{n_i}\right)^{1/2}\right]\right]\right) ,
\]
with probability $1-e^{-x},$ with the convention $\lambda_{m+1}=0$.
In particular, denoting $j_{\rho}=\max\{1\leq k\leq m: \lambda_k\geq \rho\}$ we have that
\begin{equation}
\label{random_code_grid_cell_is_adpative_in_a_range}
T(f^r_{gc},\rho)\leq \frac{16}{\left\lfloor \min_{1\leq k\leq m}\left\{\sum_{i=1}^{k}\frac{n_i}{\lambda_i}\right\} \right\rfloor}
\end{equation}
with probability $1-e^{-x}$, that is $f^r_{gc}$ is adaptive in the range of $\rho$ where the following extra conditions holds:
\begin{itemize}
    \item for all $k\leq j_\rho$,  
    \[\lambda_k \geq \frac{K_3}{\sqrt{\sum_{i<k}\frac{n_i}{\lambda_i}}},\]
    \item and also that 
    \[\rho\geq \frac{K_3}{\sqrt{\sum_{i\leq j_{\rho}}\frac{n_i}{\lambda_i}}},\]
\end{itemize}
where $K_3:=4\sqrt{2}K_2(1+x+\log(m)+(x+\log(m)^{1/2})$.
\end{proposition}

\begin{proof}
Fix $0\leq \rho\leq 1/2$, take $s_1,s_2\in\bS^{1}$ with $d(s_1,s_2)\geq \rho$ and define, as in the proof of Proposition \ref{T_min_g_code_grid_cells}, 
\[
j_{s_1,s_2}=\max\left\{i\in [n]: \frac{\lambda_i}{2}\geq d(s_1,s_2) \right\}.
\]
Lemma \ref{equal_distance_mod} ensures that for $1\leq j\leq j_{s_1,s_2}$,
 \[
 d(s_1 \ \mod \ \lambda_j, s_2 \ \mod \ \lambda_j)=d(s_1, s_2)\geq \rho.
 \] 
Hence, by Proposition \ref{prop-link-grid-place-cells} we have that
\begin{equation}
\label{proof_adpat_random_code_grid_ineq_1}
\Delta^{f^r_{gc}}_{s_1,s_2}\geq \frac{1}{2}\sum_{i=1}^{j_{s_1,s_2}}
\Delta^{f^r_{gc},i}_{s_1,s_2}\geq\frac{1}{2}\sum_{i=1}^{j_{s_1,s_2}}
\min_{u_1,u_2\in\bS^{\lambda_i}:d(u_1,u_2)\geq \rho}\Delta^{f^r_{gc},i}_{u_1,u_2}.      
\end{equation}
Next, proceeding as in the proof of Proposition \ref{random_code_adaptive_place_cells} one can show that
\begin{equation}
\label{proof_adpat_random_code_grid_ineq_2}
\min_{u_1,u_2\in\bS^{\lambda_i}:d(u_1,u_2)\geq \rho}\Delta^{f^r_{gc},i}_{u_1,u_2}\geq \frac{n_i}{2}\left(\frac{\rho}{\lambda_i}-W_{\cC_i}\right),
\end{equation}
where $W_{\cC_i}$ is the random variable given by
\[
W_{\cC_i}=\max_{C\in\cC_i}\frac{1}{n_i}\sum_{j\in M_i}\left({\bf 1}_{C_i}(\xi_{ij})-\bE\left({\bf 1}_{C_i}(\xi_{ij})\right)\right),
\]
with the set $\cC_i=\left\{C_{u_1,u_2}:u_1,u_2\in\bS^{\lambda_i}:d(u_1,u_2)\geq \rho\right\}$ and $\xi_{ij}=(A_{ij},B_{ij})\in \bS^{\lambda_i}.$
Arguing for each $W_{\cC_i}$ as it was done for $W_{\cC}$ in the proof of Proposition \ref{random_code_adaptive_place_cells}, one deduces that  if $\min_{1\leq i \leq m } \{n_i\} \geq K_1$, then for all $x_1,\ldots,x_{m}\in\bR_{>0}$,  
\begin{equation}
\label{proof_adpat_random_code_grid_ineq_3}
\min_{u_1,u_2\in\bS^{\lambda_i}: d(u_1,u_2)\geq\rho}\Delta^{f^r_{gc},i}_{u_1,u_2}\geq \frac{n_i}{2}\left(\frac{\rho}{\lambda_i}-n_i^{-1/2}K_2\left(1+x_i^{1/2}\left(1+(x_i/n_i)^{1/2}\right)\right)\right)  \ \text{for all} \ i\leq m, 
\end{equation}
with probability at least $1-\sum_{i=1}^m\exp(-x_i).$ Therefore, combining inequalities \eqref{proof_adpat_random_code_grid_ineq_1}, \eqref{proof_adpat_random_code_grid_ineq_2} and \eqref{proof_adpat_random_code_grid_ineq_3} it follows if $\min_{1\leq i \leq m } \{n_i\} \geq K^2$ then for all $1\leq k\leq m$ and all $x_1,\ldots,x_{m}\in\bR_{>0}$ 
\begin{multline*}
\min_{s_1,s_2\in\bS^{1}:d(s_1,s_2)\geq \rho,j_{s_1,s_2}=  k}\Delta^{f^r_{gc}}_{s_1,s_2}\\
\geq \frac{1}{2}\sum_{i=1}^k \frac{n_i}{2}\left(\frac{\max\left\{\rho,(\lambda_{k+1}/2)\right\}}{\lambda_i}-n_i^{-1/2}K\left(1+x_i^{1/2}\left(1+(x_i/n_i)^{1/2}\right)\right)\right)\\
\geq \frac{1}{2}\min_{1\leq k\leq m}\sum_{i=1}^k \frac{n_i}{2}\left(\frac{\max\left\{\rho,(\lambda_{k+1}/2)\right\}}{\lambda_i}-n_i^{-1/2}K\left(1+x_i^{1/2}\left(1+(x_i/n_i)^{1/2}\right)\right)\right),
\end{multline*}
with probability larger or equal than $1-\sum_{i=1}^m\exp(-x_i)$. Finally, since
\[
\min_{s_1,s_2\in\bS^{1}:d(s_1,s_2)\geq \rho}\Delta^{f^r_{gc}}_{s_1,s_2}=\min_{1\leq k\leq M}\min_{s_1,s_2\in\bS^{1}:d(s_1,s_2)\geq \rho,j_{s_1,s_2}=  k}\Delta^{f^r_{gc}}_{s_1,s_2},
\]
the result follows from the above inequality with $x_i=x+\log(m)$.

It remains to prove \eqref{random_code_grid_cell_is_adpative_in_a_range}. 
Denote $\eta_x=(1+x+\log(m)+(x+\log(m)^{1/2})$ and $\eta_{x,i}=(1+(x+\log(m)^{1/2}+(x+\log(m)/n^{1/2}_i)$, and observe that $\eta_x\geq\eta_{x,i}$ which ensures that 
\begin{align*}
\sum_{i=1}^{k-1}\frac{n_i}{4}\left(\frac{\lambda_k}{2\lambda_{i}}-n^{-1/2}_i K_2\eta_{x,i}\right)\geq \frac{1}{8}\sum_{i=1}^{k-1}\frac{n_i\lambda_k}{2\lambda_{i}}-\frac{K_2\eta_{x}}{4}\sum_{i=1}^{k-1}\frac{n^{1/2}_i}{\lambda^{1/2}_i}\lambda^{1/2}_i.
\end{align*}
By Cauchy-Schwarz inequality, we know that 
\[
\sum_{i=1}^{k-1}\frac{n^{1/2}_i}{\lambda^{1/2}_i}\lambda^{1/2}_i\leq \sqrt{\sum_{i=1}^{k-1}\frac{n_i}{\lambda_i}}\sqrt{\sum_{i=1}^{k-1}\lambda_i}\leq \sqrt{2}{\sum_{i=1}^{k-1}\frac{n_i}{\lambda_i}}, 
\]
where in the second inequality we have used that $\lambda_i\leq 2^{-(i-1)}$ for each $1\leq i\leq m$.
As a consequence, we deduce that 
\[\sum_{i=1}^{k-1}\frac{n_i}{4}\left(\frac{\lambda_k}{2\lambda_{i}}-n^{-1/2}_i K_2\eta_{x,i}\right)\geq \frac{1}{8}\sum_{i=1}^{k-1}\frac{n_i}{\lambda_i}\left(\lambda_k-\frac{2\sqrt{2}K_2\eta_x}{\sqrt{\sum_{i=1}^{k-1}n_i/\lambda_i}}\right)\geq \frac{1}{16}\sum_{i=1}^{k-1}\frac{n_i}{\lambda_i}\lambda_k,
\]
Similarly, one can show that if the second extra condition holds then we also have that
\begin{equation}
\label{ineq_2_proof_upper_bound_random_code_adp_in_a_range}
\sum_{i=1}^{j_{\rho}}\frac{n_i}{4}\left(\frac{\rho} {2\lambda_i}-n_i^{-1/2}K_2\eta_{x,i}\right)\geq \frac{1}{16} \sum_{i=1}^{j_{\rho}}\left(\frac{n_i\rho}{\lambda_i}\right).    
\end{equation}
From the last to inequalities above it is easy to deduce \eqref{random_code_grid_cell_is_adpative_in_a_range}, concluding the proof of the result.
\end{proof}

\end{example}

\subsection{What can we say about the general class $\bar{\cG}$?}
The rate 
\[ 
\frac{1}{\min_{1\leq k\leq m}\sum_{j=1}^{k}\left\lfloor \frac{2n_j}{\lambda_j} \max\left\{\frac{\lambda_{k+1}}{2},\rho\right\} \right\rfloor},\]
is very difficult to understand. What would be the best choice of $n_i$'s and $\lambda_i$'s ?

\paragraph{Minimal range of detection}

\begin{example}[Extreme Dyadic code]
\label{ex-dyadic_code}
 In the extreme dyadic code $f_d$, we have $m=n$ modules $M_1,\ldots, M_n$ where the $i$-th module $M_i=\{i\}$ has period $\lambda_i=(1/2)^{i-1}$ and whose points  $a_{ii}:=a_i,b_{ii}:=b_i\in \bS^{\lambda_i}$ are such that $\theta_{a_i}=0$ and $\theta_{b_i}=(1/2)^i$.
 In the sequel, let us denote $f_{d,ii}=f_{d,i}$ so that $f_{d}=(f_{d1},\ldots, f_{dn})$.
\end{example}

\begin{proposition}
Let $f_d$ be the extreme dyadic code. For any $\rho\geq 1/2^n$, we have $T(f_d,\rho)=1$.
\end{proposition}

\begin{proof}
First, we will show that $T(f_d,\rho, \alpha)\leq 1$. To that end, we need to show that for all $s_1,s_2\in [0,1)$ such that $d(s_1,s_2)\geq \rho$, we have $\max\left\{|I^{f_d}_{s_1}\setminus I^{f_d}_{s_2}|, |I^{f_d}_{s_2}\setminus I^{f_d}_{s_1}|\right\}\geq 1.$
Notice that by Lemma \ref{Lem:base_2}, we can write
$
(\theta_{s_i})=\sum_{j=1}^{\infty}(\theta_{s_i})_{j}/2^{j}, i\in \{1,2\}.
$

\textbf{Claim.} There exists $j\in [n]$ such that $(\theta_{s_1})_j\neq (\theta_{s_2})_j$. 

Assuming that the claim is true, we can then use again Lemma \ref{Lem:base_2} to conclude that $f_{dj}(s_1)\neq f_{dj}(s_2)$ implying that $\max\left\{|I^{f_d}_{s_1}\setminus I^{f_d}_{s_2}|, |I^{f_d}_{s_2}\setminus I^{f_d}_{s_1}|\right\}\geq 1.$

We now prove the Claim. We argue by contradiction. Suppose that 
$(\theta_{s_1})_j=(\theta_{s_2})_j$ for all $j\in [n]$. 
Suppose also that $\theta_{s_1}\leq \theta_{s_2} \leq \theta_{s_1}+1/2$ (if this is not the case we change $s_2$ by  $s_1$). 
In this case, 
\[
1/2^n\leq\rho\leq d(s_1,s_2)= |\theta_{s_1}-\theta_{s_2}|=\sum_{j=n+1}^{\infty}|((\theta_{s_1})_j-(\theta_{s_2})_j)|/2^j\leq 1/2^n=\lambda_n/2,
\]
and Lemma \ref{equal_distance_mod} implies that
\[
d(s_1 \ \mod \ \lambda_n,s_2 \ \mod \ \lambda_n)=\frac{1}{2^n}.
\]
This implies that $f_{dn}(s_1)=f_{dn}(s_1 \ \mod \ \lambda_n)\neq f_{dn}(s_2 \ \mod \ \lambda_n)=f_{dn}(s_2)$ and we can apply Lemma \ref{Lem:base_2} to deduce that $(\theta_{s_1})_n\neq (\theta_{s_2})_n$, which is absurd. This concludes the proof of the Claim. 

To conclude the proof, observe that by taking $\theta_{s_1}=0$ and $\theta_{s_2}=1/2$, we have $(\theta_{s_1})_j=0$ for all $j\geq 1$, $(\theta_{s_2})_1=1$ and $(\theta_{s_2})_j=0$ for all $j\geq 2$. By Lemma $\ref{Lem:base_2}$ we have that $f_{d1}(s_1)\neq f_{d1}(s_2)$, $f_{dj}(s_1)=f_{dj}(s_2)$ for all $2\leq j\leq n$ and $d(s_1,s_2)=|\theta_{s_1}-\theta_{s_2}|=1/2\geq \rho$. Hence, $T(f_d,\rho)=1$ and the result follows.

\end{proof}

This result needs to be put in perspective with Proposition \ref{prop:distinction_general_code} : the extreme dyadic code is able to reach the best precision that one can have, that is $1/2^n$, but it is unable to reach a faster minimax discrimination time when $\rho$ increases. So either (whatever the code in $\bar{\cG}$) one cannot achieve a faster rate than $1$ or it means that this extreme dyadic code cannot be adaptive. The following example shows a faster rate for larger $\rho$, which means that we are in the second case.


\begin{example}(Grid cells - adaptive balanced)
For a given $1 \leq m\leq n$, let $\lambda_i=2^{-(i-1)}$ and $n_i=\lfloor n/m\rfloor$ for $1\leq i\leq m-1$, and $\lambda_m=2^{-(m-1)}$ and $n_m=\lfloor n/m\rfloor+n \ \mod \ m$. We call balanced grid cells code the class $\cG_{b,m}:=\mathcal{G}((n_i,\lambda_i)_{i=1,...,m})$ corresponding to these choices
of $\lambda_i$'s and $n_i$'s. Note that each code in this class has the same number of neurons per module (except the last one).  
In what follows, let us denote $g^{b,m}_{gc}$ the code $g_{gc}$ defined in  Example \ref{ex-adpative-code-grid-cell} belonging to the class of balanced grid cells code $\cG_{b,m}$. 

\begin{proposition}
\label{balanced-grid-cells-code_adpative}
Let $1\leq m\leq n$ such that 
$n\geq 2m$. For any $\rho$ such that $1/2^m\leq \rho\leq 1/2$, \[
\frac{1}{3\lfloor n/m\rfloor}\leq T(\cG_{b,m},\rho)\leq T_{min}(g^{b,m}_{gc},\rho)\leq \frac{16}{\lfloor n/m\rfloor},
\]
that is, the code $g^{b,m}_{gc}$ is adaptive in the range $\rho \geq 1/2^m$  in the class $\cG_{b,m}$.
 \end{proposition}

\begin{proof}
Let us first compute the corresponding lower bound for $T(\cG_{b,m})$ provided by Theorem \ref{thm_lower_bound_grid_class}.
Notice that 
for all $2\leq k\leq j_{\rho}$ we have that
\[
\sum_{i=1}^{k-1}\frac{n_i\max\left\{\rho,\lambda_{k}\right\}}{\lambda_i}=\sum_{i=1}^{k-1}\frac{n_i\lambda_{k}}{\lambda_i}=\lfloor n/m\rfloor\sum_{i=1}^{k-1}2^{-i}=\lfloor n/m\rfloor\left(1-2^{-(k-1)}\right):=x_{k-1}. \]
Moreover, one can check that 
\begin{align*}
\sum_{i=1}^{j_{\rho}}\frac{n_i\max\left\{\rho,\lambda_{j_{\rho}+1}\right\}}{\lambda_i}&\geq  \rho\lfloor n/m\rfloor\sum_{i=1}^{j_{\rho}}\frac{1}{\lambda_i}\\
&=\rho\lfloor n/m\rfloor (2^{j_{\rho}}-1):=x_{j_{\rho}}.    
\end{align*}
Since $x_{k+1}-x_k=2\lfloor n/m\rfloor 2^{-(k-2)}>0$ for $1\leq k\leq j_{\rho}-2$, it follows that
\[
\min\left\{\min_{2\leq k\leq j_{\rho}}\sum_{i=1}^{k-1}\frac{n_i\max\left\{\rho,\lambda_{k}\right\}}{\lambda_i},\sum_{i=1}^{j_{\rho}}\frac{n_i\max\left\{\rho,\lambda_{j_{\rho}+1}\right\}}{\lambda_i}\right\}=\min\{x_1,x_{j_{\rho}}\}.
\]
Now, for any $1/2^m\leq \rho\leq 1/2$ it holds that 
$2^{-j_{\rho}}=\lambda_{j_{\rho}+1}\leq \rho \leq \lambda_{j_{\rho}}$ so that
$2\rho(2^{j_{\rho}}-1)\geq 2(1-\rho)\geq 2(1/2)=1$, implying that  $x_1\leq x_{j_{\rho}}$, that is  $\min\{x_1,x_{j_{\rho}}\}=x_1=(\lfloor n/m \rfloor)/2$. Hence
Theorem \ref{thm_lower_bound_grid_class} implies that
\begin{equation}
\label{proof_ineq_1lower-bound-adaptive-balanced}
T(\cG_{b,m},\rho)\geq \frac{1}{3\lfloor n/m\rfloor}.
\end{equation}
 
Next, we will compute the upper bound for $T_{min}(g^{m,b}_{gc},\rho)$ using Proposition \ref{T_min_g_code_grid_cells}.
First, observe that for all $2\leq k\leq j_{\rho}$,
\[
y_{k-1}:=\sum_{j=1}^{k-1}\left\lfloor \frac{n_j}{\lambda_j} \max\left\{\lambda_{k},\rho\right\}\right\rfloor=\left\lfloor \lfloor n/m \rfloor \frac{1}{2}\right\rfloor+\ldots+\left\lfloor \lfloor n/m \rfloor \frac{1}{2^{k-1}}\right\rfloor  
\]
so that $y_1\leq y_2<\ldots\leq y_{j_{\rho}-1}$. Let, 
\begin{align*}
y_{j_{\rho}}:=\sum_{j=1}^{j_{\rho}}\left\lfloor \frac{n_j}{\lambda_j} \max\left\{\lambda_{j_{\rho}+1}, \rho\right\}\right\rfloor&=\sum_{j=0}^{j_{\rho}-1}\left\lfloor \lfloor n/m \rfloor \rho 2^j\right\rfloor.
\end{align*}
implying that 
\[
\min_{1\leq k\leq m}\sum_{j=1}^{k}\left\lfloor \frac{n_j}{\lambda_j} \max\left\{\lambda_{k+1},\rho\right\} \right\rfloor=\min\{y_1,y_{j_{\rho}}\}.
\]
Since, 
$\lfloor n/m \rfloor /2\geq 1$ we have that $y_1\geq \lfloor n/m \rfloor /4$ (here we use that $\lfloor x\rfloor\geq x/2$ when $x\geq 1$). Moreover,
\begin{align*}
y_{j_{\rho}}&=\sum_{j=0}^{j_{\rho}-1}\left\lfloor \lfloor n/m \rfloor \rho 2^j\right\rfloor\\
&=\sum_{j=0:\lfloor n/m \rfloor \rho 2^j\geq 1}^{j_{\rho}-1}\left\lfloor \lfloor n/m \rfloor \rho 2^j\right\rfloor>(\lfloor n/m \rfloor/2)\sum_{j=0:\lfloor n/m \rfloor \rho 2^j\geq 1}^{j_{\rho}-1} \rho 2^j,
\end{align*}
so that if we denote $j_{min}=\min\{0\leq j\leq j_{\rho}-1::\lfloor n/m \rfloor \rho 2^j\geq 1\}$, then 
\[
y_{j_{\rho}}>(\lfloor n/m \rfloor/2)\sum_{j=j_{min}}^{j_{\rho}-1} \rho 2^j=(\lfloor n/m \rfloor/2)2^{j_{min}}\rho(2^{j_{\rho}-j_{min}}-1).
\]
Finally, since $\rho 2^{j_\rho}\geq 2$ and since $\lfloor n/m \rfloor\geq 2$, one can check that $j_{min}\leq j_{\rho}-2$. 
As a consequence, 
\begin{align*}
y_{j_{\rho}}&>(\lfloor n/m \rfloor/2)2^{j_{min}}\rho(2^{j_{\rho}-j_{min}}-1)\\
&\geq (\lfloor n/m \rfloor/2)(2^{j_{\rho}-j_{min}+1}/2^{j_{\rho}}-\rho)2^{j_{min}}
=(\lfloor n/m \rfloor/2)(2-\rho 2^{j_{min}})\geq \lfloor n/m \rfloor \textcolor{blue}{(1-1/4)}\\
&>\lfloor n/m \rfloor /4.
\end{align*}
Hence, we have shown that if $\lfloor n/m \rfloor\geq 2$, then
$\min\{y_1,y_{j_{\rho}}\}\geq (\lfloor n/m \rfloor/4)$ and Proposition  \ref{T_min_g_code_grid_cells} ensures that 
\[
T(\cG_{b,m},\rho)\leq T_{min}(g^{b,m}_{gc},\rho)\leq \frac{16}{\lfloor n/m\rfloor}.
\]
The result follows from   \eqref{proof_ineq_1lower-bound-adaptive-balanced} and the inequality above.
\end{proof}

This result about adaptivity is very interesting. Indeed it tells us that whatever the balanced adaptive code that we use with $m$ modules, the minimax discrimination time does not vary as a function of $\rho$ and in particular does not decrease when the distance between two stimuli increases. It stays constant at $m/n$ up to multiplicative constants. In the extreme case with $m=n$ modules, we recover the extreme dyadic code whose minimax discrimination time is 1 whatever the distance. On the other hand, and 
as a corollary of Proposition \eqref{balanced-grid-cells-code_adpative}, we have the following result.

\begin{corollary}
\label{cor_rate_tuned_and_balanced_grid_cells}
For all $2^{-n/2}\leq \rho\leq 1/2$, let $m=\lfloor \log_2(\rho^{-1}) \rfloor$. Then the code $g^{b,m}_{gc}$ satisfies 
\[
\frac{1}{6}\left\lfloor \frac{\log_2(\rho^{-1})}{n} \right\rfloor\leq \frac{1}{3\lfloor n/m\rfloor}\leq T(\cG_{b,m},\rho)\leq T_{min}(g^{b,m}_{gc},\rho)\leq 32\left\lceil\frac{\log_2(\rho^{-1})}{n}\right\rfloor.
\]
\end{corollary}

The proof of Corollary \ref{cor_rate_tuned_and_balanced_grid_cells}  is straightforward. Note that this result shows that if we know in advance at which distance $\rho$ one needs to detect, one can use a balanced adaptive code with $m=\lfloor \log_2(\rho^{-1}) \rfloor$ modules to reach a rate much faster than the place cells codes. Indeed, $\log_2(\rho^{-1})\ll \rho^{-1}$ when $\rho$ is small. On the other hand, once $m$ is fixed Proposition \ref{balanced-grid-cells-code_adpative} tells us that the minimax discrimination time is then constant and cannot decreases when $\rho$ increases as $\log_2(\rho^{-1})/n$. In this sense, balanced grid cell codes cannot be adaptive. The question remains open for unbalanced grid cells codes.

\end{example}

\hspace*{-0.5cm}\fbox{%
   \begin{minipage}{0.95\textwidth}
\begin{center}   
\textbf{\large Summary of the results on grid cells codes}
\end{center}

\begin{enumerate}
    \item One can compute the minimax discrimination rate on $\cG((n_i,\lambda_i)_{i=1,...,m}).$
    \item In general, grid cells code obtained with the adaptive code of Example 3 on each module are able to reach this minimax discrimination rate, for every $\rho$ up to a $\log_2(\rho^{-1})$ factor but on a restricted range of $\rho$.
    \item The random codes do not lose this extra $\log_2(\rho^{-1})$ factor but they have a different range of $\rho$.
    \item Extreme dyadic codes can reach the best possible precision in $2^{-n}$. However, their discrimination rate cannot be faster for larger $\rho$.
    \item Balanced grid cells code obtained with the adaptive code of Example 3 on each module  with number of modules $m=\log_2(\rho^{-1})$ are able to reach the rate $\log_2(\rho^{-1})/n$, which is faster than the corresponding minimax rate for place cells code. However, balanced grid cells code cannot go faster when the distance between the stimuli increases.
    \item In particular, we do not know the minimax discrimination rate over the general class $\bar{\cG}$ and we do not know if adaptivity is possible there. However, we do know that the extreme dyadic code which achieves the precision $2^{-n}$ or the balanced grid cells code cannot be adaptive.
    \end{enumerate}
\end{minipage}%
}

\section{Numerical illustration}

We have simulated what happens for $n=100$ cells with firing rate $\mu=30$. We used 5 different configurations, as indicated below.
\begin{itemize}
    \item Place cells - adaptive. That is the code $f$ is given by Example 3 
    \item Place cells - random. The code $f$ is picked at random as in Example 4 
    \item Grid cells - adaptive balanced. It consists in 20 modules with $\lambda_i=2^{-(i-1)}$. All the 20 modules have the same $n_i=5$. Inside each module, the code is taken as Example 3.
    \item Grid cells - adaptive decreasing. It consists in 20 modules with $\lambda_i=2^{-(i-1)}$, but the vector of $n_i$ is given by 
    $(15, 13, 11, 10,  8,  7,  6,  5,  4,  4,  3,  3,  2,  2,  2,  1,  1,  1,  1,  1)$. Inside each module, the code is taken as Example 3.
    \item Grid cells - random balanced. It consists in 20 modules with $\lambda_i=2^{-(i-1)}$. All the 20 modules have the same $n_i=5$. Inside each module, the code is taken at random as Example 4.
\end{itemize}

We picked $s=1/3$, which is not in any of the periods of the different modules. We also picked $s'=s+\rho$ with $\rho$ in a grid between $2^{-21}$ and 0.5.

\begin{figure}
    \centering
  \begin{tabular}{cc}
   \includegraphics[width=0.47\textwidth]{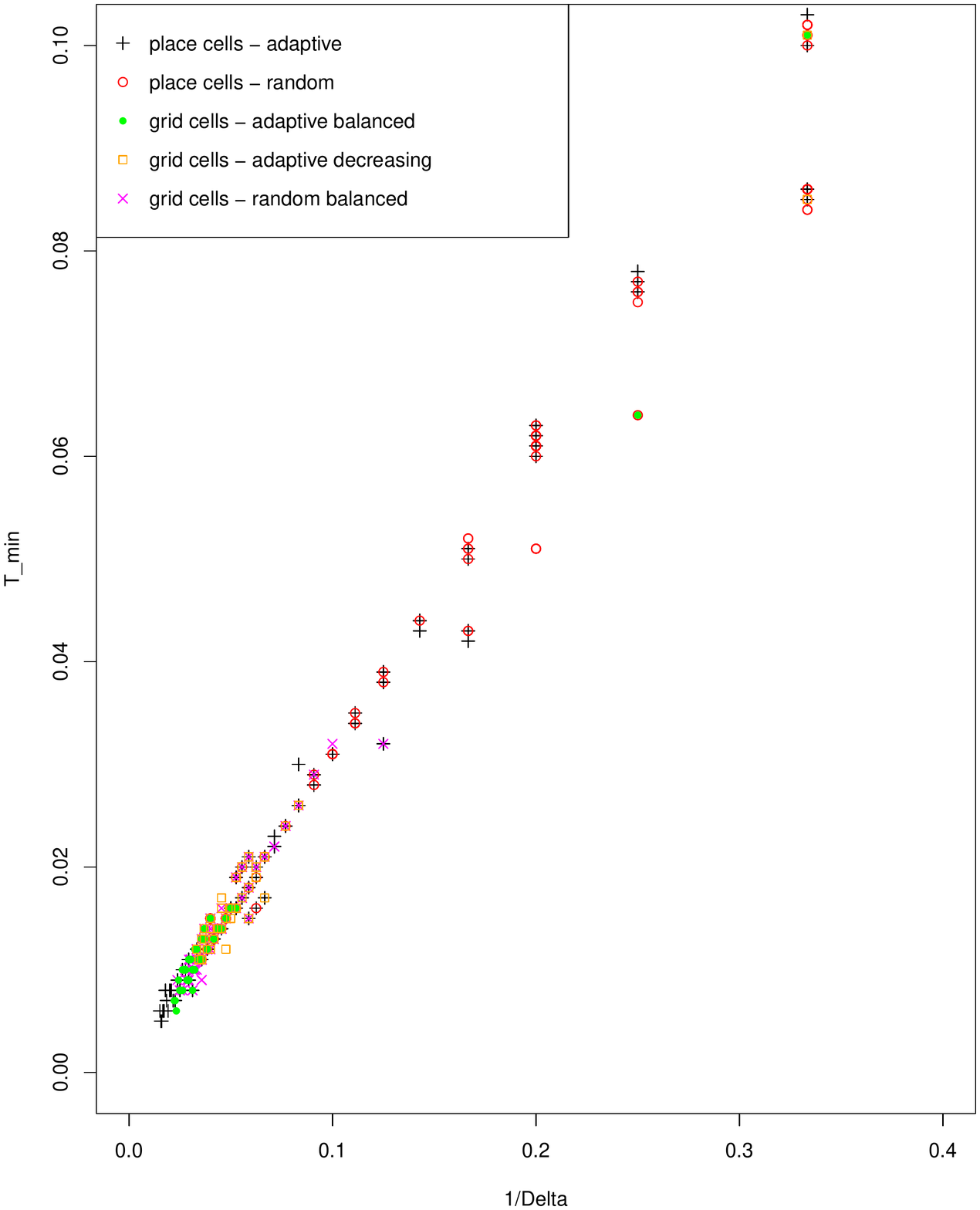}&\includegraphics[width=0.47\textwidth]{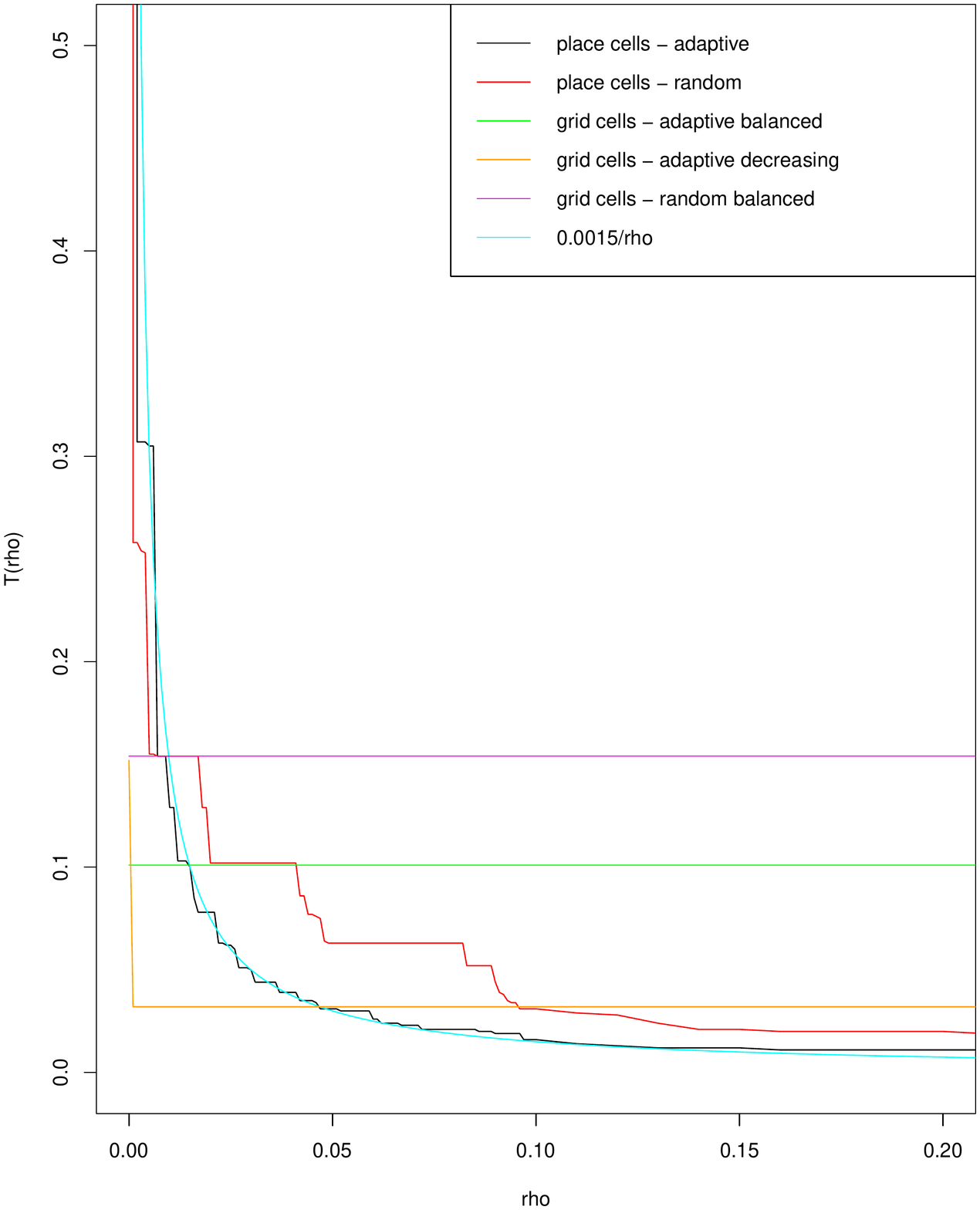}\\
    \end{tabular}
    \caption{\label{fignum} Discrimination time as a function of $\Delta^f_{s_1,s_2}$ and $\rho$. On the left, $T_{min}(f,s_1,s_2,\alpha)$ as a function of $1/\Delta^f_{s,s'}$ for the 5 different codes. On the right, $T_{min}(f,\rho,\alpha)$ as a function of $\rho$ for the 5 different codes.}
    \label{fig:my_label}
\end{figure}

For various possible $T$ in a grid from 0.001 to 20, we simulated 5000 times the test $\Psi$ given by \eqref{def:optimal_test_1} where we choose $s_1=s$ and $s_2=s'$ if $|I^f_s\setminus I^f_{s'}|\geq |I^f_{s'}\setminus I^f_{s}|$ and the reverse if this is not the case.
In each case, we found $T_{min}(f,s,s',\alpha)$ the smallest time in the grid that gives an error less than $\alpha$ (the error is evaluated by Monte Carlo method on the 5000 simulations).
In Figure \ref{fignum} on the left, we see $T_{min}(f,s,s',\alpha)$ as a function of $1/\Delta^f_{s,s'}$, with $\Delta^f_{s,s'}=\max\left\{|I^f_s\setminus I^f_{s'}|, |I^f_{s'}\setminus I^f_{s}|\right\}$. As 
derived in \eqref{Tmin}, we also see in the simulations that the same constant of proportionality holds whatever the code (this constant depending only on $\mu$ and $\alpha$) and that it was legitimate to study directly $1/\Delta^f_{s,s'}$.

It was not possible to compute $T(f,\rho)$ as defined by \eqref{Trho}, so we used the following quantity as a proxy:
\[T(f,s,\rho,\alpha)=\max_{\rho' \geq \rho \mbox{ in the grid}} T_{min}(f,s,s+\rho',\alpha),\]
for $s=1/3$.
In Figure \ref{fignum} on the right, we plotted $T(f,s,\rho,\alpha)$ as a function of $\rho$.
We see that place cells - adaptive are indeed following a curve in $1/\rho$ as expected and that place cells - random  have a similar behavior. In particular, with $n=100$ they cannot detect a $\rho$ of order $1/(2n)=0.005$. On the other hand, all grid cells can reach this range, because they can reach a much smaller range, at least  $2^{-20}$. 
Besides,  grid cells - (random or adaptive) balanced cannot have a time of detection which decreases when $\rho$ increases. Hence there is a point where the place cells system is quicker in discrimination time than the grid cells system. 

We tried a non balanced version of the grid cells (grid cells - adaptive decreasing), to make the discrimination time decreasing in $\rho$. However, as one can see on Figure \ref{fignum} on the right,  even if the discrimination time becomes decreasing in $\rho$, it reaches a plateau like behavior (up to logarithmic behavior that cannot be seen on the curve). Therefore, even this non balanced version of the grid cells  is still slower than the place cells system.

\section{Conclusion, discussion and Perspectives}
We have adopted a new point of view on rate coding : the testing point of view and use it to enhance differences between a place cell system and a grid cell system with the same number of neurons.

\paragraph{On the testing point of view.} The testing point of view is complementary to the estimation/information theory point of view developed originally by \cite{Brunel_1998} for place cells or other cells with simple receptor field. Indeed,  for place cells, Brunel and Nadal proved in a similar framework as the random code of the present work that the Fisher information is proportional to  $n$, the number of cells. The only main difference is that they used triangular codes $f$ instead of step functions with two parameters. If the Fisher information is proportional to $n$, it means that the standard deviation of the best estimator of the stimulus $s$ is in $n^{-1/2}$. We have proved that a similar random place cell system can discriminate between two positions as long as they are at least at distance of order $n^{-1/2}$. In this sense, both frameworks seem consistent. However, we have also proved (see Example 3) that another place cells code achieves a much smaller precision in $1/n$. Moreover the testing set-up allows us to capture an interesting phenomenon, which should take place in practice and might be tested via experiments. When two stimuli are very different, the discrimination time between both should be smaller.

\paragraph{On the place cell system.} We have shown that adaptivity in terms of $\rho$ (the distance between two stimuli) is possible. It means that not only certain codes have the ability to discriminate faster when the stimuli are further away but it also means that the rate $1/(n\rho)$ is up to constant not improvable by a place cell system. This is reached not only by very particular codes such as the one of Example 3 but also by random codes, for $\rho \geq c/\sqrt{n},$ for some positive constant $c$.

\paragraph{On the grid cell system.} As shown already by the study of the Fisher information \cite{Fiete_2008,Fiete_2011}, we can prove that grid cells can reach a precision that is much smaller than the one of place cells and which is exponentially decreasing with $n$. We have also shown that the rate $2^{-n}$ for the discrimination time is in fact an absolute lower bound for all kind of codes with only two values and that the grid cells in this sense are the ones able to achieve the smallest precision not only with respect to place cells, but with respect to any code. We have also been able to derive upper and lower bound on the minimal discrimination time for grid cells with a given number of cells per module and given period for each module. In particular, the distribution of the cells inside a module is derived from their place cell equivalent and a random uniform distribution inside a module leads to the best rate up to constant as soon as $\rho$ is large enough. This is coherent with the experiment of \cite{stensola_et_al_12} that has found no structure in the distribution of the center of the receptor field inside a given module.
Finally, for a fix $\rho$, particular balanced grid cells code are able to achieve the discrimination rate of $\log(1/\rho)/n$ which is much faster than place cells, but we have not been able to find a fix grid cells code whose discrimination rate would decrease when $\rho$ increases in an interesting range (that is past the detection range of the first module). Informally once $\rho<1/(2n_1)$, it seems that all codes are limited by what happens in the first module, that is all discrimination rates seem to be equal to $1/(2n_1)$. However, we have not been able to rigorously prove this fact.

\paragraph{Place cells versus grid cells.} As a consequence, and contrary to place cells, we have not been able to find a general grid cell adaptive code that would achieve a minimax discrimination rate that can only be smaller or equal to $\log(1/\rho)/n$ thanks to our results. Our tentative  non balanced codes as shown in the simulation were not satisfying. We do not know if adaptation is possible for general grid cells. If it is not possible, then there is definitely an advantage to have both systems (grid cells and place cells) at the same time. The simulations show a  compromise that can be made in using both systems at once (see Figure \ref{fignum} on the right). Indeed, grid cells can go at much smaller precision that place cells but are much slower when $\rho$ increase than place cells. In this sense, having a combination of both systems would allow a fast reaction time when stimuli are far away and a good reaction time even if the stimuli are very close.

\paragraph{Limitations and Open problems.}
We restricted ourselves to one dimensional, periodic stimuli (or circular maze). We do not know at the moment how to generalize it to larger dimension, especially if we want grid cells to have an hexagonal pattern. Moreover, \cite{stensola_et_al_12} have proved that the periods progress in a geometric manner (see also \cite{Wei_2015} for similar geometric progressions in 1d) but that the ratio is not an integer and so far our method is relying too strongly on periodicity to allow that. To go beyond this, we would need to consider stimuli with boundaries effect and if the adaptation to boundaries for grid cells have been described \cite{Krupic_2015} modeling precisely the boundary effect from a mathematical point of view is 
not straightforward.
Also one might want to add more realistic rate functions $f(s)$ that step functions, but we do not think this would massively change the rates we found as long as the shape of $f(s)$ is not authorized to vary much.
Finally, even in the 1d circular case, the problem of adaptivity and even the computation of the minimax discrimination rate of the general grid cells codes remains open. It means that we do not know what is the best choice in terms of discrimination rate for the number of cells per module, the scales or even the number of modules.

\section{Auxiliary Results}
In this section, we collect all auxiliary results used throughout the paper. We start with the following proposition.
  
\begin{lemma}
\label{Prop:Poi_Upper_Lower}
Let $X$ be a Poisson random variable with parameter $\theta>0.$ For any $x\geq 0$,
\begin{equation}
\label{Prop:Poi_upper}
\bP\left(X\geq \theta(1+x)\right)\leq \exp\left\{-\frac{\theta x^2}{2(1+x/3)}\right\},
\end{equation}
and for any $0\leq x\leq \theta,$
\begin{equation}
\label{Prop:Poi_lower}
\bP\left(X\leq \theta-x\right)\leq \exp\left\{-\frac{x^2}{2\theta}\right\}.
\end{equation} 
\end{lemma}

\begin{proof}
Let $\psi(\lambda)=e^{\lambda}-\lambda-1$ for $\lambda\in\bR$ 
and notice that $\bE(e^{\lambda (X-\theta)})=e^{\theta\psi(\lambda)}$.
One can check that for all $0\leq \lambda<3$, we have
\[
\psi(\lambda)\leq \frac{\lambda^2}{2(1-\lambda/3)}.
\]
so that for any $t>0$,
\[
\bP(X\geq \theta+t)\leq \exp\left\{-\sup_{\lambda\in (0,3)}\left(t\lambda -\frac{\lambda^2\theta}{2(1-\lambda/3)}\right)\right\}=\exp\{-9 \theta h_1(t/3\theta)\},
\]
where $h_1(x)=1+x-\sqrt{1+2x}, x>0$. Since $h_1(x)\geq x^2/(2(1+x))$ it follows that for any $t>0,$
\[
\bP(X\geq \theta+t)\leq \exp\left\{-\theta \frac{(t/\theta)^2}{2(1+(t/3\theta))}\right\}.
\]
Taking $t=\theta x$ we obtain \eqref{Prop:Poi_upper}.


To prove inequality \eqref{Prop:Poi_lower}, one can use the fact that  $\psi(\lambda)\leq \lambda^2/2 $ for $\lambda\leq 0,$ to deduce for any $0\leq x\leq \theta$,
\[
\bP(X\leq \theta-x)\leq \exp\{-\sup_{\lambda>0}(\lambda x-\theta\lambda^2/2)\}=\exp\{-x^2/2\theta\},
\]
concluding the proof.
\end{proof}


\begin{lemma}
\label{equal_distance_mod}
Let $0<\lambda\leq 1 $. For all $s_1,s_2\in\bS^{1}$ such that $d(s_1,s_2)\leq \lambda/2$, it holds that $d(s_1,s_2)=d(s_1 \ \mod \ \lambda, s_2 \ \mod \ \lambda).$
\end{lemma}
\begin{proof}
It is sufficient to analyse the case $s_1,s_2\in \bS^{1}$ are such that 
$0\leq \theta{s_1}\leq 1/2,$   $\theta_{s_1}\leq \theta_{s_2}\leq \theta_{s_1}+1/2$ and $d(s_1,s_2)\leq \lambda/2$.
Write $\theta_{s_1}=k\lambda+r_1$ where $0\leq r_1=\theta_{s_1} \ \mod \ \lambda<\lambda.$

Suppose that $r_1+\lambda/2<\lambda$. In this case, we must that $r_1<\lambda/2$ and $\theta_{s_2}=k\lambda+r_2$ with $r_1\leq r_2=\theta_{s_2} \ \mod \ \lambda\leq r_1+\lambda/2$. Hence 
\[
\theta_{s_1 \ \mod \ \lambda}=\theta_{s_1} \ \mod \ \lambda =r_1\leq r_2=\theta_{s_2} \ \mod \ \lambda= \theta_{s_2 \ \mod \ \lambda}\leq r_1+\lambda/2
\]
so that
\[
d(s_1 \ \mod \ \lambda, s_2 \ \mod \ \lambda )=(\theta_{s_2 \ \mod \ \lambda}-\theta_{s_1 \ \mod \ \lambda})=(r_2-r_1)=d(s_1,s_2).\]

Now, it remains to consider the case $r_1+\lambda/2\geq \lambda$. In this case, we have that $r_1\geq \lambda/2$. If $\theta_{s_2}=k\lambda+r_2$ then we must have $r_1\leq r_2<\lambda$ so that
$d(s_1 \ \mod \ \lambda, s_2 \ \mod \ \lambda )=(\theta_{s_2 \ \mod \ \lambda}-\theta_{s_1 \ \mod \ \lambda})=r_2-r_1=d(s_1,s_2)$. Otherwise, $\theta_{s_2}=(k+1)\lambda +r_2$ where $r_2\leq r_1+\lambda/2-\lambda=r_1-\lambda/2$ so that 
$d(s_1 \ \mod \ \lambda, s_2 \ \mod \ \lambda )=r_2+\lambda-r_1=d(s_1,s_2)$ and the result follows.
\end{proof}

\begin{lemma}
\label{mult_distance_equal_mod}
Let $0<\sigma_1\leq \sigma_2\leq 1/2 $ such that $\sigma_2=p\sigma_1$ and $1=q\sigma_1$ where $p,q\in \bN$ are such that $q\geq p.$ 
For all $s_1,s_2\in\bS^{1}$ such that $d(s_1,s_2)=\sigma_2$, it holds that $s_1 \ \mod \ \sigma_1=s_2 \ \mod \ \sigma_1.$
\end{lemma}
\begin{proof}
Suppose that $\theta_{1}\leq \theta_{2}$ and write $\theta_{s_1}=k\sigma_1+r$ where $0\leq r= \theta_{s_1} \ \mod \ \sigma_1<\sigma_1$. 
If  $0\leq \theta_{s_2}\leq \theta_{s_1}+1/2$ and $d(s_2,s_1)=\sigma_2$, then $\theta_{s_2}-\theta_{s_1}=d(s_2,s_1)=\sigma_2=p\sigma_1$ so that
$\theta_{s_2}=(p+k)\sigma_1+r$, implying that 
$\theta_{s_2} \ \mod \ \sigma_1=r=\theta_{s_1} \ \mod \ \sigma_1$.
Since $\theta_{s_2 \ \mod \ \sigma_1}=\theta_{s_2} \ \mod \ \sigma_1=\theta_{s_1} \ \mod \ \sigma_1=\theta_{s_1 \ \mod \ \sigma_1}$ it follows that $s_1 \ \mod \ \sigma_1=s_2 \ \mod \ \sigma_1.$

Now, if $1/2+\theta_{s_1}\leq \theta_{s_2}<1$ and $d(s_2,s_1)=\sigma_2$, then $1-(\theta_{s_2}-\theta_{s_1})=d(s_2,s_1)=\sigma_2=p\sigma_1$ so that $\theta_{s_2}=1+(k-p)\sigma_1+r=(q+k-p)\sigma_1$ and once again $\theta_{s_2} \ \mod \ \sigma_1=r=\theta_{s_1} \ \mod \ \sigma_1$, and the result follows.
\end{proof}

\begin{lemma}
\label{Lem:base_2}
For each $x\in [0,1)$, let $x_1=\lfloor 2x \rfloor$ and $x_{j}=\left\lfloor 2^{j}(x-\sum_{k=1}^{j-1}x_k/2^k)\right\rfloor $ for $j\geq 2.$ Then $x_j\in \{0,1\}$ for all $j\in\bN_{>0}$, $x=\sum_{j=1}^{\infty}x_j/2^j$ and $0\leq \sum_{k=j}^{\infty}x_k/x^k<1/2^{j-1}$ for any $j\in\bN_{>0}$. In particular, if $f_d=(f_{d,1},\ldots, f_{d,n})$ denotes the dyadic code defined in Example \ref{ex-dyadic_code}, then  for any $i\in [n]$, we have $f_{di}(s)=\mu$ if and only if $(\theta_{s})_i=0$.
\end{lemma}
\begin{proof}
It is clear that $x_1\in \{0,1\}$. To show that $x_j\in \{0,1\}$, first observe that denoting $u:=x-\sum_{k=1}^{j-1}x_k/2^k$, we have $u-\lfloor 2^j u\rfloor/2^j\in [0,1/2^j).$ 
Since $\lfloor 2^j u\rfloor=x_j$, it follows then that
$x-\sum_{k=1}^{j}x_k/2^k=x-\lfloor 2^j u\rfloor/2^j\in [0,1/2^j)$. This implies that
$x=\sum_{j=1}^{\infty}x_k/2^k$
and also that
$0\leq 2^{j+1}(x-\sum_{k=1}^{j}x_k/2^k)=x-\lfloor 2^j u\rfloor/2^j)<2$ which ensures that $x_{j+1}\in\{0,1\}$ and also that $\sum_{k=j+1}^{\infty}x_k/2^k<2^{-j}$. 

It remains to show that $f_{di}(s)=\lambda$ if and only if $s_i=0$. To that end, first observe that $f_{di}(s)=f_{di}(s \ \mod \ \lambda_i)=\mu$ if and only if $\theta_{s}-\lfloor 2^{i-1}\theta_s\rfloor/2^{i-1}=\theta_{s} \ \mod \ \lambda_i=\theta_{s \ \mod \ \lambda_i} <1/2^i$. Now, since $\theta_{s}-\lfloor 2^{i-1}\theta_s\rfloor/2^{i-1}=\sum_{j=i+1}^{\infty}(\theta_{s})_j/2^j+(\theta_{s})_i/2^i$, it follows that $f_{di}(s)=\mu$ implies that $\sum_{j=i+1}^{\infty}(\theta_{s})_j/2^j+(\theta_{s})_i/2^i<1/2^i$, forcing $(\theta_{s})_i=0.$ On the other hand, if $(\theta_{s})_i=0$ then  
$\sum_{j=i+1}^{\infty}(\theta_{s})_j/2^j+(\theta_{s})_i/2^i=\sum_{j=i+1}^{\infty}(\theta_{s})_j/2^j<1/2^i$, where the inequality holds by the first part of the proof. Recalling that $\sum_{j=i+1}^{\infty}(\theta_s)_j/2^j+(\theta_s)_i/2^i=\theta \ \mod \ \lambda_i=\theta_{s \ \mod \ \lambda_i}$, it follows that $(\theta_s)_i=0$ implies that $f_{di}(s)=\mu$, concluding the proof. 
\end{proof}

\begin{lemma}
\label{prob_of_unif_interval_having_2_specific_points}
Let $A,B$ be independent and uniformly distributed points in $\bS^{\lambda}$ with $\lambda\in\bR_{>0}$. For any pair $s_1,s_2\in \bS^{\lambda}$, we have that
\[
\bP\left(s_1,s_2\in \llb A, B\llb \right)=\lambda^{-2}\left(\frac{\lambda^2}{2}-|\theta_{s_2}-\theta_{s_1}|\left(\lambda-|\theta_{s_2}-\theta_{s_1}|\right)\right).
\]
\end{lemma}
\begin{proof}
Suppose that $\theta_{s_1}\leq \theta_{s_2}$ and note that
\begin{multline*}
\bP\left(s_1,s_2\in \llb A, B\llb \right)=\bP\left(\theta_A\leq \theta_{s_1},\theta_{s_2}< \theta_B<\lambda\right)+\bP\left( \theta_{s_1}\leq \theta_A<\theta_{s_2}, \theta_{s_1}<\theta_B<\theta_{A}\right)\\+\bP\left(\theta_B<\theta_A\leq \theta_{s_1} \right)+\bP\left(\theta_{s_2}\leq \theta_A<\lambda, \theta_{s_2}<\theta_{B}< \theta_A\right).
    \end{multline*}
We now compute each one of the probabilities appearing in the right-hand side of the equality above. Since $\theta_A$ and $\theta_B$ are independent and uniformly distributed on $[0,\lambda)$, one can check that
\[
\begin{cases}
\bP\left(\theta_A\leq \theta_{s_1},\theta_{s_2}< \theta_B<\lambda\right)=\frac{\theta_{s_1}}{\lambda}\frac{(\lambda-\theta_{s_2})}{\lambda},\\
\bP\left(\theta_B<\theta_A\leq \theta_{s_1} \right)=\frac{1}{\lambda^2}\int_{0}^{\theta_{s_1}}(\theta_{s_1}-\theta_b)d\theta_b=\theta^2_{s_1}/(2\lambda^2),\\
\bP\left( \theta_{s_1}<\theta_B< \theta_A<\theta_{s_2}\right)=\frac{1}{\lambda^{2}}\int_{\theta_{s_1}}^{\theta_{s_2}}(\theta_{s_2}-\theta_{b})d\theta_b=(\theta_{s_2}-\theta_{s_1})^2/(2\lambda^{2}),\\
\bP\left(\theta_{s_2}< \theta_B<\theta_A<\lambda,\right)=\frac{1}{\lambda^{2}}\int_{\theta_{s_2}}^{\lambda}(\lambda-\theta_{b})d\theta_b=(\lambda-\theta_{s_2})^2/(2\lambda^2.)
\end{cases}
\]
Summing these fours terms, we obtain that 
\begin{align*}
\bP\left(s_1,s_2\in \llb A, B\llb \right)&=\lambda^{-2}\left(\theta_{s_1}(\lambda-\theta_{s_2})+\theta^2_{s_1}/2+(\theta_{s_2}-\theta_{s_1})/2+(\lambda-\theta_{s_2})^2/2\right)\\
=&\lambda^{-2}\left(\lambda^2/2-\lambda(\theta_{s_2}-\theta_{s_1})+(\theta_{s_2}-\theta_{s_1})^2\right)\\
=& \lambda^{-2}\left(\lambda^2/2-|\theta_{s_2}-\theta_{s_1}|(\lambda-|\theta_{s_2}-\theta_{s_1}|)\right),
\end{align*}
and the result follows.
\end{proof}

\bibliographystyle{amsplainhyper_m}  
\bibliography{bibliomna}

\ACKNO{
This research  was supported by the French government, through CNRS, the UCA$^{Jedi}$
and 3IA
C\^ote d’Azur Investissements d’Avenir managed by the National Research
Agency (ANR-15 IDEX-01 and ANR-19-P3IA-0002), directly by the ANR
project ChaMaNe (ANR-19-CE40-0024-02), by the interdisciplinary Institute for Modeling in Neuroscience and Cognition (NeuroMod) of Universit\'e C\^ote d'Azur and by the FAPESP project {\em Research, Innovation and
Dissemination Center for Neuromathematics} (grant 2013/07699-0).
G.O. thanks FAPERJ (grant E-26/201.397/2021) and CNPq (grants 303166/2022-3 and 432310/2018-5) for financial support.
GO also thanks the support of the Institut Henri Poincaré,  LabEx CARMIN and Institut des Hautes Études Scientifiques (grants UAR 839 CNRS-Sorbonne Université and ANR-10-LABX-59-01) during his stay in France where substantial parts of this research were done.

The authors thank Fabrizio Capitano for fruitful discussions on place cells and grid cells.
}

\nocite{*}
\end{document}